\documentclass[11pt]{amsart}
\usepackage{amsmath,amssymb,amsthm,latexsym,cite,cancel}
\usepackage[small]{caption}
\usepackage{graphicx,wasysym,overpic,tikz,color}
\usepackage{subfigure}
\usepackage{cite}
\usepackage[colorlinks=true,urlcolor=blue,
citecolor=red,linkcolor=blue,linktocpage,pdfpagelabels,
bookmarksnumbered,bookmarksopen]{hyperref}
\usepackage[italian,english]{babel}
\usepackage{units}
\usepackage{enumitem}
\usepackage[left=2.1cm,right=2.1cm,top=2.71cm,bottom=2.71cm]{geometry}
\usepackage[hyperpageref]{backref}
\usepackage{float}

\usepackage[colorinlistoftodos]{todonotes}
\newtheorem{theorem}{Theorem}[section]
\newtheorem{definition}[theorem]{Definition}

\newtheorem{lemma}[theorem]{Lemma}
\newtheorem{remark}[theorem]{Remark}

\newtheorem{corollary}[theorem]{Corollary}

\newcommand{\abs}[1]{\lvert#1\rvert}
\newcommand{\norm}[1]{\lVert#1\rVert}
\newcommand{\red}[1]{\textcolor{red}{#1}}
\newcommand{\blue}[1]{\textcolor{blue}{#1}}
\newcommand{\D}{\mathcal{D}}
\newcommand{\G}{\mathcal{G}}
\newcommand{\K}{\mathcal{K}}
\newcommand{\h}{\mathcal{H}}
\tikzstyle{nodo}=[circle,draw,fill,inner sep=0pt,minimum size=%
1.5mm]

\numberwithin{equation}{section}

\title[Nonlinear Dirac equations]{Normalized solutions of nonlinear Dirac equations on noncompact metric graphs with localized nonlinearities}

\author[Z. He]{Zhentao He}

\address[Z. He]{\newline\indent
	School of Mathematics
	\newline\indent
	East China University of Science and Technology
	\newline\indent
	Shanghai 200237, PR China }
\email{\href{mailto:hezhentao2001@outlook.com}{hezhentao2001@outlook.com}}

\author[C. Ji]{Chao Ji}

\address[C. Ji]{\newline\indent
	School of Mathematics
	\newline\indent
	East China University of Science and Technology
	\newline\indent
	Shanghai 200237, PR China }
\email{\href{mailto:jichao@ecust.edu.cn}{jichao@ecust.edu.cn}}

\subjclass[]{}
\date{\today}

\begin{document}

\begin{abstract}
In this paper, we study the following nonlinear Dirac equations (NLDE) on noncompact metric graph $\G$ with localized nonlinearities
\begin{equation}
    \D u - \omega u= a\chi_\K\abs{u}^{p-2}u,
\end{equation}
where $\D$ is the Dirac operator on $\G$, $u: \G \to \mathbb{C}^2$, $\omega\in \mathbb{R}$, $a > 0$, $\chi_\K$ is the characteristic function of the compact core $\K$, and $p>2$. First, for $2<p<4$,  we prove the existence of normalized solutions to (NLDE) using a perturbation argument. Then, for $p \geq 4$, we establish the assumption under which normalized solutions to (NLDE) exist. Finally, we extend these results to the case $a<0$ and, for all $p>2$, prove the existence of normalized solutions to (NLDE) when $\lambda = -mc^2$ is an eigenvalue of the operator $\D$. In the Appendix, we study the influence of the parameters $m, c > 0$ on the existence of normalized solutions to (NLDE). To the best of our knowledge, this is the first study to  investigate the normalized solutions to (NLDE) on metric graphs.
\end{abstract}
\maketitle


\section{Introduction}

Throughout the paper, we consider a connected noncompact metric graph $\mathcal{G} = (\mathrm{V}, \mathrm{E})$, where $\mathrm{E}$ is the set
of edges and $\mathrm{V}$ is the set of vertices. We assume that $\mathcal{G}$ has a finite number of edges and the compact
core $\K$ is non-empty. Unbounded edges are identified with (copies of) $\mathbb{R}^+ =  [0, +\infty )$ and are called half-lines, while bounded edges are identified with closed and bounded intervals $I_e = [0, \ell_e]$, $\ell_e > 0$.  If $\G$ is a metric graph with a finite number of vertices, its compact core $\K$ is defined as the metric subgraph of $\G$ consisting of all the bounded edges of $\G$. A connected metric graph has the natural structure of a locally compact metric space, the metric being given by the shortest distance measured along
the edges of the graph. Moreover, a connected metric graph with a finite number of vertices is compact if and only if it contains no half-line.

 Quantum graphs (metric graphs equipped with differential operators) arise naturally as simplified models in mathematics, physics, chemistry, and engineering when one considers propagation of waves of various nature through a quasi-one-dimensional (e.g., meso- or nano-scale) system that looks like a thin neighborhood of a graph. For further details on Quantum graphs, one may refer to \cite{Be}.
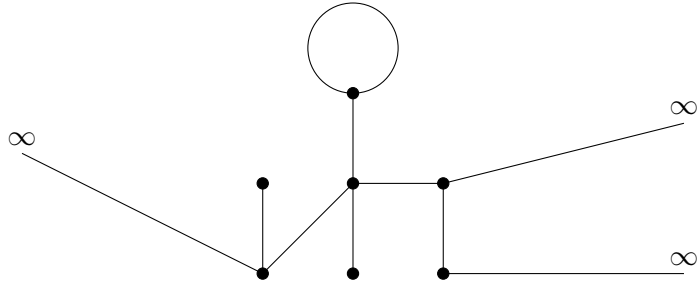
\begin{figure}[H]
	\begin{tikzpicture}[xscale= 0.4,yscale=0.4]
		\draw (-14,3)--(-11,6) ; \draw (0,8)--(-8,6);
		\draw (-11,6)--(-8,6); \draw (-22,7)--(-14,3);
				\draw (0,3)--(-8,3);
		\draw (-11,9)--(-11,6);
		\draw (-14,3)--(-14,6);
		\draw (-11,3)--(-11,6);
		\draw (-8,3)--(-8,6);
		\node at (-14,3) [nodo] {} ; \node at (-14,6) [nodo] {};
		\node at (-11,3) [nodo] {}; \node at (-11,6) [nodo] {} ;
		\node at (-8,3) [nodo] {}; \node at (-8,6) [nodo] {};
		\node at (-11,9) [nodo] {};
  \draw (-22,7.5)  node{$\infty$};
  	\draw (0,8.5)  node{$\infty$};
     	\draw (0,3.5)  node{$\infty$};
   \draw (-11,10.5) circle (1.5);
		\end{tikzpicture}
	\caption{A noncompact graph $\G$ with finitely number of edges and a non-empty compact core.}
\end{figure}

Currently, a considerable attention has been devoted to the following nonlinear Schr\"odinger equations (NLSE) on metric graphs
\begin{equation}\label{eqsinvolvet}
     \imath\partial_t \psi = -\psi'' -\abs{\psi}^{p-2}\psi,
\end{equation}
where $p > 2$  and $\partial_t=\frac{\partial}{\partial t}$.
The physical motivations for studying (NLSE) on metric graphs
are related to the qualitative behavior of Bose-Einstein condensates in ramified traps and nonlinear optics in Kerr media. The stationary solutions to \eqref{eqsinvolvet} , that is, $\psi(t, x) = e^{-\imath \lambda t}u(x)$, with $\lambda \in \mathbb{R}$, that solve
\begin{equation}
\label{eqs}
    -u'' - \lambda u= \abs{u}^{p-2}u.
\end{equation}
The initial study of equation \eqref{eqs} focused on the infinite $N$-star graphs, see \cite{Ad1,Ad2}. For results regarding the existence and multiplicity of normalized solutions to equation \eqref{eqs} on compact metric graphs, we refer the readers to \cite{Jea,Do,Ch}. In this paper, we focus on the results on noncompact metric graphs. In the $H^1$-subcritical case (i.e., $2<p < 6$) and $H^1$-critical case (i.e., $p=6$), the existence, multiplicity and uniqueness of normalized solutions to equation \eqref{eqs} on noncompact metric graphs have been extensively studied, we refer the readers to \cite{Do0,Do1,Ad3,Ad4,Ad5,Ad6,Pi}. Specifically, in \cite{Ad3}, Adami, Serra and Tilli established a topological assumption on noncompact metric graphs that guarantees the non-existence of normalized ground states for equations \eqref{eqs} with $2<p<6$ and proved the existence of normalized ground states on some specific noncompact metric graphs. In \cite{Do1}, Dovetta, Serra and Tilli investigated the uniqueness and non-uniqueness of normalized ground states to equation \eqref{eqs} with $2<p\leq6$, they obtained the uniqueness on two specific noncompact metric graphs and constructed a metric graph that admits at least two normalized ground states of same mass having different Lagrange multipliers. Moreover, for the $H^1$-supercritical case (i.e., $p>6$), the existence of normalized solutions to equation \eqref{eqs} has started to be considered in \cite{Ch} on compact metric graphs. In \cite{Do0}, Dovetta, Jeanjean and Serra proved the existence of normalized solutions to equation \eqref{eqs} on noncompact metric graphs satisfying some topological assumptions for $p > 6$, using monotonicity trick, energy estimates and discussions related to Morse index type information.

For any noncompact metric graph $\G$ with a finite number of edges and a non-empty compact
core $\K$, \cite{Gn,No} introduced the following modification of equation \eqref{eqs}, which assumes the nonlinearity affects only the non-empty compact core $\K$
\begin{equation}
\label{eqsloc}
    -u'' - \lambda u= \chi_\K\abs{u}^{p-2}u,
\end{equation}
where $\chi_\K$ is the characteristic function of the compact core $\K$. For $2<p<6$, the existence and multiplicity of normalized solutions to equation \eqref{eqsloc} have been studied in \cite{Se,Se1,Te}. In \cite{Do2}, Dovetta and Tentarelli showed that, for $p=6$, the existence of normalized ground states to equation \eqref{eqsloc} depends on the topology of $\G$. For $p > 6$, the existence of normalized solutions to equation \eqref{eqsloc} was proved in \cite{Bort} using monotonicity trick, blow-up analysis and discussions on Morse index type information. Furthermore, by employing a more general blow-up analysis, the multiplicity  of normalized solutions to equation \eqref{eqsloc} for $p > 6$ was proved in \cite{Ca}.

However, the energy functionals associated with Dirac equations are strongly indefinite. Thus, the methods previously developed  for on (NLSE) cannot be directly applied to nonlinear Dirac equations (NLDE).

The Dirac operator on metric graphs is denoted by
\begin{equation}
\label{eqdef}
    \D := -\imath c\frac{d}{dx} \otimes \sigma_1 +mc^2\otimes\sigma_3
\end{equation}
where $m > 0$ represents the mass of the generic particle of the system and $c > 0$ represents the speed of light, and $\sigma_1, \sigma_3$ are Pauli matrices, i.e.,
\begin{equation}
\label{eqdefsig}
\sigma_1:=\left(\begin{array}{ll}
0 & 1 \\
1 & 0
\end{array}\right) \quad \text { and } \quad \sigma_3:=\left(\begin{array}{cc}
1 & 0 \\
0 & -1
\end{array}\right).
\end{equation}

Consider the following (NLDE) on metric graphs
\begin{equation}\label{eqdinvolvet}
     \imath\partial_t \psi = \D\psi + \abs{\psi}^{p-2}\psi,
\end{equation}
where $p > 2$. The physical motivation for such a model primarily comes from solid state physics and nonlinear optics, see \cite{Had,Tr} and references therein.
\cite{Sa} suggests studying the stationary solutions to equation \eqref{eqdinvolvet} on $3$-star graph, that is, $\psi (t, x) = e^{- \imath\omega t} u(x)$, with $\omega  \in \mathbb{R}$, leading to the equation
\begin{equation}
\label{eqdexnon}
    \D u - \omega u= \abs{u}^{p-2}u.
\end{equation}
In \cite{Bo}, Borrelli, Carlone and Tentarelli considered the case of a localized nonlinearity, that is
\begin{equation}
\label{eqdnoa}
    \D u - \omega u= \chi_\K\abs{u}^{p-2}u.
\end{equation}
More precisely, they proved that for every $\omega \in (-mc^2,mc^2)$, there exists infinitely many (distinct) pairs of bound states (arising as critical points of the corresponding action functional) of frequency $\omega$ of \eqref{eqdnoa} and, which converge to the bound states of \eqref{eqs} in the nonrelativistic limit, namely as $c \to \infty$, when $2<p<6$. For the Cauchy problem associated with equations \eqref{eqdinvolvet} on noncompact metric graphs and the existence of solutions to equations \eqref{eqdexnon} on star graphs, one can refer to \cite{Bo1}.

Our focus here is on normalized solutions of \eqref{eqdnoa}, i.e., solutions satisfying $\int_\G\abs{u}^2\,dx = \rho$ with $\rho > 0$.  For simplicity, we assume $\int_\G\abs{u}^2\,dx = 1$. With a slight modification, our results can be extended to the other vallues of $\rho > 0$.

In \cite{Ding}, Ding, Yu and Zhao studied (NLDE) in $\mathbb{R}^3$, and they showed the existence of normalized solitary wave solutions of (NLDE), in the $H^\frac{1}{2}$-subcritical case (for $\mathbb{R}^3$, it is $2<p<\frac{8}{3}$), using a reduction and  perturbation argument.  Firstly, they introduced the perturbation functional $I_{r,\mu}$ associate with (NLDE) in $\mathbb{R}^3$ and the corresponding reduced functional $J_{r,\mu}$, where $r > 1$ and $\mu>0$ are parameters. Secondly, they proved that the reduced functional possesses the mountain geometry structure, constructed test-functions by Fourier transform, and gave an upper bound estimate for the frequency $\omega$ with the test-functions. Thirdly, they showed, for $\mu>0$ small enough and  $r>0$ large enough, there exists a critical point sequence $\{u_{r,\mu}\}$ of $J_{r,\mu}$ and $\{u_{r,\mu}\}$ possesses a subsequence which converges to a non-trivial solution to (NLDE) as $r \to \infty$ and $\mu \to 0^+$. Finally, they established a non-existence result, and then obtained the existence of normalized solutions.

To the best of our knowledge, the existence of normalized solutions for nonlinear Dirac equations (NLDE) on metric graphs remains unexplored in the current literature. This gap motivates us to explore this class of problems and establish new results in this direction.  Inspired by \cite{Bo,Bu,Ding}, in this paper, we study the following (NLDE) on noncompact metric graphs
\begin{equation}
\label{eq1}
    \D u - \omega u= a\chi_\K\abs{u}^{p-2}u,
\end{equation}
where $p > 2$, $a > 0$ and $\chi_\K$ is the characteristic function of the compact core $\K$. Specifically, in the $H^\frac{1}{2}$-subcritical case (for metric graphs, it is $2<p<4$), we prove the existence of normalized solutions to equation \eqref{eq1} for $a$ small enough in Theorem \ref{th2} and for all $a>0$ in Theorem \ref{thalla} when $\G$ satisfies some certain topological assumption. Moreover, we study the condition when there exists a normalized solutions of the nonlinear Dirac equation in the $H^\frac{1}{2}$-critical and supercritical cases (for metric graphs, they are $p=4$ and $p > 4$) in Thoerem \ref{th3}, and, for $4\leq p <6$, we construct some noncompact metric graphs that admit normalized solutions to equation\eqref{eq1} for $a>0$ small enough, in Theorem \ref{th4} and Corollary \ref{co1}. We also extend the results above to the case $a<0$ (see equation \eqref{eq-a}) in Theorem \ref{th-a1} and \ref{th-a2}, then, for all $p>2$, we prove the existence of normalized solutions to equation \eqref{eq-a} for $a$ small enough in Theorem \ref{th-a3} when $\lambda = -mc^2$ is an eigenvalue of the operator $\D$ (we provide the topological assumption on $\G$ in Remark \ref{remarknegative} under which $\lambda = -mc^2$ is an eigenvalue of $\D$). Finally, we study the influence of the parameters $m, c > 0$ on the existence of normalized solutions to equations \eqref{eq1} and \eqref{eq-a}.

Due to the absence of Fourier transform on general metric graphs, the method used to construct test-functions in $\mathbb{R}^3$, which is crucial to give an upper bound estimate for the frequency $\omega$, is no longer applicable. Therefore, we propose a new approach for  constructing test-functions without replying on Fourier transform. Besides, in the prove of Theorem \ref{thalla}, we also need to overcome the difficulty that the corresponding spectrum results in \cite{Bu} are missing for Dirac operators. Moreover, in $H^\frac{1}{2}$-critical and supercritical cases, we cannot establish a non-existence result as in $H^\frac{1}{2}$-subcritical case. However, by imposing an additional assumption on the energy of solutions, we deprive a non-existence result for $p \geq 4$ (see Lemma \ref{lempgeq4}), which is essential for the proof of Theorem \ref{th4}, Corollary \ref{co1}, Theorem \ref{th-a2} and Theorem \ref{th-a3}.

Now, we recall some basic settings of Dirac equations on metric graphs, as introduced in \cite{Bo,Bo1}.
Consistently, a function $u: \G \to \mathbb{C}$ is actually a family of functions $(u_e)$, where $u_e: I_e \to \mathbb{C}$ is the restriction of $u$ to the edge $e$. The usual $L^p$ spaces on the metric graph  are defined as follows
as
$$
L^p(\mathcal{G}):=\bigoplus_{e \in \mathrm{E}} L^p(I_e)
$$
with norms
$$
\norm{u}_{L^p(\mathcal{G})}^p := \sum_{e \in \mathrm{E}}\norm{u_e}_{L^p(I_e)}^p, \quad \text { if } p \in[1, \infty), \quad \text { and } \quad \norm{u}_{L^{\infty}(\mathcal{G})}:=\max _{e \in \mathrm{E}}\norm{u_e}_{L^{\infty}(I_e)},
$$
while the Sobolev spapces $H^m(\G)$ are defined as by
by
$$
H^m(\mathcal{G}):=\bigoplus_{e \in \mathrm{E}} H^m(I_e)
$$ with norm

$$
\norm{u}_{H^m(\G)}^2=\sum_{i=0}^m\norm{u^{(i)}}_{L^2(\mathcal{G})}^2.
$$
Accordingly, a spinor $u = (u^1, u^2)^T: \G \to \mathbb{C}^2$ is a family of $2$-spinors
$$
u_e=\binom{u_e^1}{u_e^2}: I_e \longrightarrow \mathbb{C}^2, \quad \forall e \in \mathrm{E},
$$
and thus
$$
L^p(\mathcal{G}, \mathbb{C}^2):=\bigoplus_{e \in \mathrm{E}} L^p(I_e, \mathbb{C}^2),
$$
endowed with the norms
$$
\norm{u}_{L^p(\G, \mathbb{C}^2)}^p := \sum_{e \in \mathrm{E}}\norm{u_e}_{L^p(I_e, \mathbb{C}^2)}^p, \quad \text { if } p \in[1, \infty), \quad \text { and } \quad \norm{u}_{L^{\infty}(\G, \mathbb{C}^2)}:=\max_{e \in \mathrm{E}}\norm{u_e}_{L^{\infty}(I_e, \mathbb{C}^2)},
$$
whereas
$$
H^m(\mathcal{G}, \mathbb{C}^2):=\bigoplus_{e \in \mathrm{E}} H^m(I_e, \mathbb{C}^2)
$$
with the norms
$$
\|u\|_{H^m(\G, \mathbb{C}^2)}^2:=\sum_{e \in \mathrm{E}}\norm{u_e}_{H^m(I_e, \mathbb{C}^2)}^2.
$$
For convenience, we often abbreviate $\norm{u}_{L^p(\G, \mathbb{C}^2)}$ as $\norm{u}_p$.

Note that, in the case of (NLSE), the condition that $u$ is continuous on $\G$ is often contained in $H^1(\G)$. However, in the case of (NLDE), we shall keep the conditions of spinors at the vertices separate, as defiened in the following.
\begin{definition}[{\cite[Definition 2.3]{Bo}}]
\label{defD}
Let $\mathcal{G}$ be a metric graph and  $m, c>0$. We call the Dirac operator with Kirchhoff-type vertex conditions the operator $\mathcal{D}: L^2(\mathcal{G}, \mathbb{C}^2) \rightarrow L^2(\mathcal{G}, \mathbb{C}^2)$ with action
\begin{equation}\label{eqdefde}
    \mathcal{D}_{\left.\right|_{I_e}} u=\mathcal{D}_e u_e:=-\imath c \sigma_1 u_e^{\prime}+m c^2 \sigma_3 u_e \quad \forall e \in \mathrm{E},
\end{equation}
$\sigma_1, \sigma_3$ being the matrices defined in \eqref{eqdefsig}, and the domain of $D$ given by
$$
\operatorname{dom}(\mathcal{D}):=\{u \in H^1(\mathcal{G}, \mathbb{C}^2): u \text { satisfies \eqref{eqdefdom1} and \eqref{eqdefdom2}}\}
$$
where
\begin{equation}
\label{eqdefdom1}
    u_e^1(\mathrm{v})=u_f^1(\mathrm{v}) \quad \forall e, f \succ \mathrm{v}, \quad \forall \mathrm{v} \in \mathcal{K},
\end{equation}
\begin{equation}
\label{eqdefdom2}
    \sum_{e \succ v} u_e^2(\mathrm{v})_{ \pm}=0 \quad \forall \mathrm{v} \in \mathcal{K},
\end{equation}
"$e \succ v$" meaning that the edge $e$ is incident at the vertex v and $u_e^2(\mathrm{v})_{ \pm}$standing for $u_e^2(0)$ or $-u_e^2(\ell_e)$ according to whether $x_e$ is equal to $0$ or $\ell_e$ at $v$.
\end{definition}
We recall some basic properties of the Dirac operator$\D$, as discussed in \cite{Bo}. The operator $\D$  is self-adjoint on $L^2(\mathcal{G}, \mathbb{C}^2)$. In addition, and its spectrum is given by
\begin{equation}
\label{eqsp}
    \sigma{(\D)}=( - \infty , - mc^2
] \cup [mc^2, +\infty ).
\end{equation}
Next, we define the associated quadratic form $\mathcal{Q}_{\mathcal{D}}$ and its domain $\operatorname{dom}(\mathcal{Q}_{\mathcal{D}})$ as follows,
$$
\operatorname{dom}(\mathcal{Q}_{\mathcal{D}}):=\left\{u \in L^2(\mathcal{G}, \mathbb{C}^2): \int_{\sigma(\mathcal{D})}|\nu| d \mu_u^{\mathcal{D}}(\nu)\right\}, \quad \mathcal{Q}_{\mathcal{D}}(u):=\int_{\sigma(\mathcal{D})} \nu d \mu_u^{\mathcal{D}}(\nu),
$$
where $\mu_u^{\mathcal{D}}$ denotes the spectral measure associated with $\mathcal{D}$ and $u$. Additionally, $\operatorname{dom}(\mathcal{Q}_{\mathcal{D}})$ is a closed subspace of
$$
H^{1 / 2}(\mathcal{G}, \mathbb{C}^2):=\bigoplus_{e \in \mathrm{E}} H^{1 / 2}(I_e) \otimes \mathbb{C}^2
$$
with respect to the norm induced by $H^{1 / 2}(\mathcal{G}, \mathbb{C}^2)$. As a consequence of Sobolev embeddings, it follows that
\begin{equation*}
    \operatorname{dom}(\mathcal{Q}_{\mathcal{D}}) \hookrightarrow L^p(\G,\mathbb{C}^2) \quad \forall p \in [2, \infty )
\end{equation*}
and that, in addition, the embedding $\operatorname{dom}(\mathcal{Q}_{\mathcal{D}}) \hookrightarrow L^p(\K,\mathbb{C}^2)$ is compact, for $p \in [2, \infty )$, due to the compactness
of $\K$. For simplicity, we denote $\operatorname{dom}(\mathcal{Q}_{\mathcal{D}})$ by $Y$.

It is possible to define the inner product on $Y$ by
\begin{equation*}
    (u,v) =\Re(\abs{\D}^\frac{1}{2}u,\abs{\D}^\frac{1}{2}v)_2,
\end{equation*}
where $\Re$ denotes the real part of a complex number, and the induced norm is given
\begin{equation*}
    \norm{u} = (u,u)^\frac{1}{2}.
\end{equation*}
From \eqref{eqsp}, for any $u \in Y$, we have
$$\norm{u}^2 \geq mc^2\norm{u}_2^2.$$

According to \eqref{eqsp}, the space $L^2(\G, \mathbb{C}^2)$ possesses the orthogonal decomposition:
\begin{equation}
     L^2(\G, \mathbb{C}^2) = L^- \oplus L^+.
\end{equation}
Moreover, we can decompose the form domain $Y$ as the orthogonal sum of the positive and negative spectral subspaces for the operator $\D$,
i.e.,
\begin{equation}
    Y = Y^- \oplus Y^+ \text{, where } Y^\pm = L^\pm \cap Y.
\end{equation}
As a result, every $u \in Y$ can be written as
\begin{equation*}
    u = P^+u +P^-u =: u^+ + u^-,
\end{equation*}
where
$P^\pm$ are the orthogonal projectors onto $Y^\pm$.

\begin{theorem}
\label{th1}
    Let $\G$ be any noncompact metric graph with a non-empty compact core $\K$, and let $2< p < 4$. Then, we have\\
    \begin{enumerate}[label=\rm(\roman*)]
        \item either there exists a non-trivial $u \in \operatorname{dom}(\D)$ (as defined in Definition \ref{defD}) and $\omega \in (-mc^2, mc^2)$ such that
    \begin{equation*}
    \left\{\begin{aligned}
        &\D_e u_e - \omega u_e= a\chi_\K\abs{u_e}^{p-2}u_e \quad \forall e\in \mathrm{E},\\
        &\int_\G\abs{u}^2\,dx = 1.
    \end{aligned}
    \right.
\end{equation*}
\item  or for all $\omega \in (-mc^2, mc^2)$, there exists a non-trivial $u_\omega \in \operatorname{dom}(\D)$ such that
    \begin{equation*}
    \left\{\begin{aligned}
        &\D_e u_{\omega,e} - \omega u_{\omega,e} = a\chi_\K\abs{u_{\omega,e}}^{p-2}u_{\omega,e} \quad\forall e\in \mathrm{E},\\
        &\int_\G\abs{u_\omega}^2\,dx < 1.
    \end{aligned}
    \right.
    \end{equation*}
    \end{enumerate}
\end{theorem}
Define
\begin{equation}
    \label{eqa0}
a_0 = \frac{m^\frac{4-p}{2}c^{4-p}}{2C_{p,\K}},
\end{equation}
where $C_{p,\K}$ denote the constant in the Gagliardo-Nirenberg-Sobolev inequality (see Lemma \ref{lemgnsg}).
\begin{theorem}\label{th2}
    Under the assumptions of Theorem \ref{th1} and $0<a<a_0$,  the second
case of Theorem \ref{th1} cannot occur.
Specially, there exists a non-trivial $u \in \operatorname{dom}(\D)$ (as defined in Definition \ref{defD}) and $\omega \in (-mc^2, mc^2)$ such that
    \begin{equation*}
    \left\{\begin{aligned}
        &\D_e u_e - \omega u_e= a\chi_\K\abs{u_e}^{p-2}u_e \quad \forall e\in \mathrm{E},\\
        &\int_\G\abs{u}^2\,dx = 1.
    \end{aligned}
    \right.
\end{equation*}
\end{theorem}
Next, we examine the conditions under which normalized solutions to  equation \eqref{eq1} for all $a>0$ and for all $p\geq4$.
\begin{theorem}\label{thalla}
Under the assumptions of Theorem \ref{th1}, suppose that $\K$ is a tree with at most one leaf incident with no half-line in $\G$ (see Definition \ref{defcomg}). Then, for all $a>0$, there exists a non-trivial $u \in \operatorname{dom}(\D)$ (defined by Definition \ref{defD}) and $\omega \in (-mc^2, mc^2)$ such that
    \begin{equation*}
    \left\{\begin{aligned}
        &\D_e u_e - \omega u_e= a\chi_\K\abs{u_e}^{p-2}u_e \quad \forall e\in \mathrm{E},\\
        &\int_\G\abs{u}^2\,dx = 1.
    \end{aligned}
    \right.
\end{equation*}
\end{theorem}
\begin{theorem}
    \label{th3}
    Let $\G$ be any noncompact metric graph with a non-empty compact core $\K$, and let $p \geq 4$. Then for all $\omega_*\in(-mc^2, mc^2)$, there exists $a^*_{\omega_*}>0$ such that, for all $a>a^*_{\omega_*}$, we have\\
    \begin{enumerate}[label=\rm(\roman*)]
        \item either there exists a non-trivial $u \in \operatorname{dom}(\D)$ (as defined in Definition \ref{defD}) and $\omega \in [\omega_*, mc^2)$ such that
    \begin{equation*}
    \left\{\begin{aligned}
        &\D_e u_e - \omega u_e= a\chi_\K\abs{u_e}^{p-2}u_e \quad \forall e\in \mathrm{E},\\
        &\int_\G\abs{u}^2\,dx = 1.
    \end{aligned}
    \right.
\end{equation*}
\item  or for all $\omega \in [\omega_*, mc^2)$, there exists a non-trivial $u_\omega \in \operatorname{dom}(\D)$ such that
    \begin{equation*}
    \left\{\begin{aligned}
        &\D_e u_{\omega,e} - \omega u_{\omega,e} = a\chi_\K\abs{u_{\omega,e}}^{p-2}u_{\omega,e} \quad\forall e\in \mathrm{E},\\
        &\int_\G\abs{u_\omega}^2\,dx < 1.
    \end{aligned}
    \right.
    \end{equation*}
    \end{enumerate}
\end{theorem}
Set
\begin{equation}\label{eqa*0}
    a_{*,0} := 2^{-\frac{p}{4}}C_{p,\G}^{-1}m^\frac{4-p}{2}c^{4-p}\left(\frac{p}{p-2}\right)^\frac{-p^2+5p-4}{2p-4},
\end{equation}
where $C_{p,\G}$ denote the constant in the Gagliardo-Nirenberg-Sobolev inequality(see Lemma \ref{lemgnsg}).
\begin{remark}\label{remarkp>4}
    Under the assumptions of Theorem \ref{th3}, let $w_* = 0$. If $0 < a^*_{0} < a_{*,0}C_{p,\G}C_{p,\K}^{-1}$, then, as shown in the proof of Theorem \ref{th4}, by Lemma \ref{lempgeq4}, for all $a\in (a^*_{0},a_{*,0}]$, the first case of Theorem \ref{th3} holds. That is, there exists a non-trivial $u \in \operatorname{dom}(\D)$ (as defined in Definition \ref{defD}) and $\omega \in [0, mc^2)$ such that
    \begin{equation*}
    \left\{\begin{aligned}
        &\D_e u_e - \omega u_e= a\chi_\K\abs{u_e}^{p-2}u_e \quad \forall e\in \mathrm{E},\\
        &\int_\G\abs{u}^2\,dx = 1.
    \end{aligned}
    \right.
\end{equation*}
 However, the relationship $a^*_{0} < a_{*,0}$ is not guaranteed for all noncompact metric graphs and $p \geq 4$.
\end{remark}
For $4 \leq p <6$, we can construct some noncompact metric graphs that admit normalized solution for nonlinear Dirac equation \eqref{eq1}. Let $\G$ be any noncompact metric graph having a non-empty compact core $\K$, and let $\mathcal{H}$ denote a half-line of $\G$. For any $\ell \geq 0$, define $\mathcal{H}_\ell:=[0,\ell]\subset\mathcal{H}$, where we tacitly identified the edge $\mathcal{H}$ with $[0,+\infty)$. The characteristic function  $\chi_{\mathcal{H}_\ell}: \G \to \mathbb{R}$ is defined as
\[
 \chi_{\mathcal{H}_\ell} := \begin{cases}1 & \text { for } x \in \mathcal{H}_\ell\\ 0 & \text { for } x \in \G\backslash\mathcal{H}_\ell. \end{cases}
\]
 Let $\chi_{\ell} := \chi_{\K}+\chi_{\mathcal{H}_\ell}$.
 \begin{theorem}\label{th4}
    Let $\G$ be a noncompact metric graph with a non-empty compact core $\K$, and assume $4 \leq p <6$. Then, for all $\omega_*\in(-mc^2, mc^2)$ and $a > 0$, there exists $\ell^*(\omega_*,a)>0$ such that, for all $\ell>\ell^*(\omega_*,a)$, the following holds:\\
    \begin{enumerate}[label=\rm(\roman*)]
        \item either there exists a non-trivial $u \in \operatorname{dom}(\D)$ (as defined in Definition \ref{defD}) and $\omega \in [\omega_*, mc^2)$ such that
    \begin{equation*}
    \left\{\begin{aligned}
        &\D_e u_e - \omega u_e= a\chi_\ell\abs{u_e}^{p-2}u_e \quad \forall e\in \mathrm{E},\\
        &\int_\G\abs{u}^2\,dx = 1.
    \end{aligned}
    \right.
\end{equation*}
\item  or for all $\omega \in [\omega_*, mc^2)$, there exists a non-trivial $u_\omega \in \operatorname{dom}(\D)$ such that
    \begin{equation*}
    \left\{\begin{aligned}
        &\D_e u_{\omega,e} - \omega u_{\omega,e} = a\chi_\ell\abs{u_{\omega,e}}^{p-2}u_{\omega,e} \quad\forall e\in \mathrm{E},\\
        &\int_\G\abs{u_\omega}^2\,dx < 1.
    \end{aligned}
    \right.
    \end{equation*}
    \end{enumerate}
Moreover, if $0<a<a_{*,0}$, then for all $\ell>\ell^*(0,a)$, the first case (for $\omega_* =0$) occurs, that is, there exists a non-trivial $u \in \operatorname{dom}(\D)$ (as defined in Definition \ref{defD}) and $\omega \in [0, mc^2)$ such that
    \begin{equation*}
    \left\{\begin{aligned}
        &\D_e u_e - \omega u_e= a\chi_\ell\abs{u_e}^{p-2}u_e \quad \forall e\in \mathrm{E},\\
        &\int_\G\abs{u}^2\,dx = 1.
    \end{aligned}
    \right.
    \end{equation*}
\end{theorem}
Let the vertex incident to the edge $\mathcal{H}$ in $\G$ be denoted by $v_\mathcal{H}$. To construct the desired graph $\G_\ell$, we add  a new vertex $v_\ell$ to $\G$ at $x_{\mathcal{H}} = \ell$ on $\mathcal{H}$, where we  identify the edge $\mathcal{H}$ with $[0,+\infty)$. The new edge $v_\mathcal{H}v_\ell$ is identified with the internal $[0,\ell]$, and $x_{v_\mathcal{H}v_\ell}$ is equal to $0$ at $v_\mathcal{H}$. If $u$ is a solution of equation
\begin{equation}\label{eqell}
    \D u - \omega u= a\chi_\ell\abs{u}^{p-2}u
\end{equation}
on $\G$, then we can construct a solution $u_\ell:\G_\ell \to \mathbb{C}^2$ of equation \eqref{eq1} on $\G_\ell$.  Denote the set of edges of $\G_\ell$ by $\mathrm{E}_\ell$, the set of vertices of $\G_\ell$ by $\mathrm{V}_\ell$, the compact core of $\G_\ell$ by $\K_\ell$, and Dirac operator on $\G_\ell$ by $\D_\ell$ (as defined in Definition \ref{defD}). The following corollary is a direct consequence of Theorem \ref{th4} for the normalized solutions of (NLDE) on $G_\ell$.

\begin{corollary}\label{co1}
    Under the assumptions of Theorem \ref{th4}, then, for all $\omega_*\in(-mc^2, mc^2)$, $a > 0$ and $\ell>\ell^*(\omega_*,a)$, the following holds:\\
    \begin{enumerate}[label=\rm(\roman*)]
        \item either there exists a non-trivial $u \in \operatorname{dom}(\D_\ell)$ (as defined in Definition \ref{defD}) and $\omega \in [\omega_*, mc^2)$ such that
    \begin{equation*}
    \left\{\begin{aligned}
        &\D_{\ell,e} u_e - \omega u_e= a\chi_{\K_\ell}\abs{u_e}^{p-2}u_e \quad \forall e\in \mathrm{E}_\ell,\\
        &\int_{\G_\ell}\abs{u}^2\,dx = 1.
    \end{aligned}
    \right.
    \end{equation*}
    \item  or for all $\omega \in [\omega_*, mc^2)$, there exists a non-trivial $u_\omega \in \operatorname{dom}(\D_\ell)$ such that
    \begin{equation*}
    \left\{\begin{aligned}
        &\D_{\ell,e} u_{\omega,e} - \omega u_{\omega,e} = a\chi_{\K_\ell}\abs{u_{\omega,e}}^{p-2}u_{\omega,e} \quad\forall e\in \mathrm{E}_\ell,\\
        &\int_{\G_\ell}\abs{u_\omega}^2\,dx < 1.
    \end{aligned}
    \right.
    \end{equation*}
    \end{enumerate}
Moreover, if $0<a<a_{*,0}$, then for all $\ell>\ell^*(0,a)$, the first case (for $\omega_* =0$) occurs, that is, there exists a non-trivial $u \in \operatorname{dom}(\D_\ell)$ and $\omega \in [0, mc^2)$ such that
    \begin{equation*}
    \left\{\begin{aligned}
        &\D_{\ell,e} u_e - \omega u_e= a\chi_{\K_\ell}\abs{u_e}^{p-2}u_e \quad \forall e\in \mathrm{E}_\ell,\\
        &\int_{\G_\ell}\abs{u}^2\,dx = 1.
    \end{aligned}
    \right.
    \end{equation*}
\end{corollary}

\begin{remark}
     It is clear that, for any weak solution $u \in \operatorname{dom}(\D)$ of equation \eqref{eq1}, we have $u_e \in C^1(I_e,\mathbb{C}^{2})$ for all $e \in \mathrm{E}$ (see, e.g., \cite[Remark 6 in Chapter 8]{Ha}). However, this may not hold for weak solutions $u \in \operatorname{dom}(\D)$ of equation \eqref{eqell} with $\ell > 0$.
\end{remark}

The rest of the paper is organized as follows. In Section \ref{sect2}, we present a Gagliardo-Nirenberg-Sobolev inequality on metric graphs, introduce the perturbation functional and provide some useful lemmas. In Section \ref{sect3}, we give the proof of Theorem \ref{th1}. In Section \ref{sect4}, we establish a non-existence result for $2<p<4$, and then prove Theorem \ref{th2}. In Section \ref{sectalla}, we give some definitions in graph theory,  prove Theorem \ref{thalla}, and derive a spectrum result. In Section \ref{sect5}, we prove Theorem \ref{th3}. In Section \ref{sect6}, we establish a non-existence result for $p \geq 4$, and then prove Theorem \ref{th4} and Corollary \ref{co1}. In Section \ref{sectnegative}, we extend these results to the case $a < 0$, provide the topological assumption on $\G$ in Remark \ref{remarknegative} under which $\lambda = -mc^2$ is an eigenvalue of the operator $\D$, and prove Theorem \ref{th-a3}. In the Appendix \ref{appa}, we study the influence of the parameters $m, c > 0$ on the existence of normalized solutions to (NLDE).
\section{Preliminaries}\label{sect2}
In this section, we establish a Gagliardo-Nirenberg-Sobolev inequality for metric graphs, introduce the perturbation functional, and provide some fundamental results that will be essential for proving our main theorems.
\subsection{Gagliardo-Nirenberg-Sobolev inequality}
The following lemma is an immediate corollary of $Y \subset H^{\frac{1}{2}}(\G,\mathbb{C}^2)$ and the Gagliardo-Nirenberg-Sobolev inequality \cite[Theorem 1.1]{Ha2}.

\begin{lemma}\label{lemgnsg}
Let $p \geq 2$, there exist constants $C_{p,\K}$ and $C_{p,\G}$ that depend only on $c$, $m$, $p$ and $\G$ such that
\[
    \int_{\K} \abs{u}^p\,dx \leq C_{p,\K} \norm{u}^{p-2}\norm{u}_2^{2}, \quad \forall u \in Y,
\]
and
\[
 \int_{\K} \abs{u}^p\,dx \leq \int_{\G} \abs{u}^p\,dx \leq C_{p,\G} \norm{u}^{p-2}\norm{u}_2^{2}, \quad \forall u \in Y.
\]
\end{lemma}
\subsection{The Perturbation Functional}\label{subsecper} For the remainder of the paper, we set $c = 1$, as its value does not affect the results.

We define the energy functional $I_{\omega}: Y \to \mathbb{R}$ associated with equation \eqref{eq1} by
$$
I_{\omega}(u) := \frac{1}{2}\norm{u^{+}}^{2}-\frac{1}{2}\norm{u^{-}}^{2}-\frac{\omega}{2} \int_{\G}\abs{u}^{2} \,dx-\Psi(u)
$$
for $u=(u^{+}+u^{-}) \in Y$, where
$$
\Psi(u) := \frac{a}{p} \int_{\K}|u|^{p}\,dx.
$$
It follows from standard arguments that $I_{\omega} \in C^{2}(Y, \mathbb{R})$. Moreover, by \cite[Proposition 3.1]{Bo}, if $u \in Y$ is a critical point of $I_{\omega}$, then $u$ is a solution of problem \eqref{eq1}, and belonging to $\operatorname{dom}(\D) \subset H^1(\G, \mathbb{C}^2)$. Thus, $u_e \in C^1(I_e,\mathbb{C}^{2})$ for all $e \in \mathrm{E}$ (see, e.g., \cite[Remark 6 in Chapter 8]{Ha}).

Define $J: Y^{+} \rightarrow \mathbb{R}$ by
$$
J(v)=I_0(v+h(v))=\max \{I_{0}(v+w): w \in Y^{-}\},
$$
where $h \in C^1\left(Y^{+}, Y^{-}\right)$ and the critical points of $J$ and $I_0$ are in one-to-one correspondence via the injective map $u \rightarrow u+h(u)$ from $Y^{+}$ into $Y$, as shown in \cite{Ding,DingX,Ac}.

For any $\mu>0$, we consider operator $T_{\mu}: Y \rightarrow L^{2}(\G, \mathbb{C}^{2})$ defined by
$$
T_{\mu} u := (I+\mu\abs{\D}) u=u+\mu \D(u^{+}-u^{-}).
$$
Next, we define the following perturbation functional
$$
I_{r, \mu}(u) := \frac{1}{2}\norm{u^{+}}^{2}-\frac{1}{2}\norm{u^{-}}^{2}-\Psi(u)-H_{r, \mu}(u),\; u \in U_{\mu},
$$
where $U_{\mu} := \{u \in Y: (T_{\mu} u, u)_{2}<1\}$, and $H_{r, \mu}(u)$ is a penalization term defined by
$$
H_{r, \mu}(u) := f_{r}((T_{\mu} u, u)_{2}) \text { with } f_{r}(s) := \frac{s^{r}}{1-s} \; (0 \leq s<1),
$$
and $r>1$ is a parameter that will be chosen large enough.

The following are some properties of the penalization $H_{r, \mu}(u)$ for any $u \in U_{\mu}$ and $v, v_{1}, v_{2} \in Y$,
$$
\langle H_{r, \mu}^{\prime}(u), v\rangle=2 f_{r}^{\prime}((T_{\mu} u, u)_{2})(T_{\mu} u, v)_{2}
$$
with
$$
f_{r}^{\prime}(s)=\frac{r s^{r-1}}{1-s}+\frac{s^{r}}{(1-s)^{2}}>\frac{r}{s} f_{r}(s)>0 \;(0<s<1),
$$
and
\begin{equation*}
\langle H_{r, \mu}^{\prime \prime}(u) v_{1}, v_{2}\rangle=2 f_{r}^{\prime}((T_{\mu} u, u)_{2})(T_{\mu} v_{1}, v_{2})_{2}+4 f_{r}^{\prime \prime}((T_{\mu} u, u)_{2})(T_{\mu} u, v_{1})_{2}(T_{\mu} u, v_{2})_{2}
\end{equation*}
with
$$
f_{r}^{\prime \prime}(s)=\frac{r(r-1) s^{r-2}}{1-s}+\frac{2 r s^{r-1}}{(1-s)^{2}}+\frac{2 s^{r}}{(1-s)^{3}}>0 \;(0<s<1).
$$
In particular, $H_{r, \mu}(u)$ is convex in $u$ and
$$
\langle H_{r, \mu}^{\prime}(u), u\rangle \geq 2 r H_{r, \mu}(u),\,\,\text{for all}\,\,u \in U_{\mu}.
$$

Following the arguments in \cite[Section 2.3]{Ding}, we can apply a variational reduction of $I_{r, \mu}$, which leads to a $C^1$ reduction map $h_{r, \mu}: Y^{+} \cap U_{\mu} \rightarrow Y$ and a reduced functional $J_{r, \mu}: Y^{+} \cap U_{\mu} \rightarrow \mathbb{R}$ satisfying
\begin{equation*}
J_{r, \mu}(v) := I_{r, \mu}(v+h_{r, \mu}(v))=\max \{I_{r, \mu}(v+w): w \in Y^{-},(T_{\mu}(v+w), v+w)_{2}<1\}.
\end{equation*}
The following Lemma shows that the functional $J_{r, \mu}$ possesses the mountain pass geometry.
\begin{lemma}
The functional $J_{r, \mu}$ possesses the following properties:
  \begin{enumerate}[label=\rm(\roman*)]
\item There exist $\alpha, \rho>0$ such that $J_{r, \mu}(v) \geq \alpha$ for $v \in Y^{+} \cap U_{\mu}$ with $\|v\|=\rho$;
\item There exists  $e \in Y^{+} \cap U_{\mu}$ with $\|e\|>\rho$ such that $J_{r, \mu}(e)<0$.
\end{enumerate}
\end{lemma}
\begin{proof}
(i) For any $v \in Y^{+} \cap U_{\mu}$ with $\|v\|=\rho$, where $\rho$ is small enough, we have
$$
(T_{\mu} v, v)_{2} \leq \frac{1+m \mu}{a}\|v\|^{2}=\frac{1+m \mu}{a} \rho^{2}<1.
$$
Thus, by the definition of $J_{r, \mu}$, the monotonicity of $f_{r}(s)$ with respect to $s$, and the Sobolev inequality, we obtain
$$
\begin{aligned}
J_{r, \mu}(v) \geq I_{r, \mu}(v) & =\frac{1}{2}\|v\|^{2}-\frac{a}{p} \int_{\K}|v|^{p}\,dx-f_{r}((T_{\mu} v, v)_{2}) \\
& \geq \frac{1}{2}\|v\|^{2}-\frac{a}{p} \int_{\K}|v|^{p} \,dx-f_{r}\left(\frac{1+m \mu}{a}\|v\|^{2}\right) \\
& \geq \rho^{2}\left(\frac{1}{2}-C \rho^{p-2}-\frac{(\frac{1+m \mu}{a})^{r} \rho^{2 r-2}}{1-\frac{1+m \mu}{a} \rho^{2}}\right).
\end{aligned}
$$
Since $p>2$ and $r>1$,  we can choose $\rho>0$ small enough such that
$$
J_{r, \mu}(v) \geq \alpha \,\,\text {for }\,\, v \in Y^{+} \cap U_{\mu} \text { with }\|v\|=\rho.
$$
(ii) Choosing $v_{0} \in Y^{+}$ with $(T_{\mu} v_{0}, v_{0})_{2}=1$, it is easy to verify that
$$
\lim _{t \rightarrow 1^{-}} J_{r, \mu}(t v_{0})=-\infty.
$$
Therefore, there exists $t_{0}<1$ such that $\norm{t_{0} v_{0}}>\rho$ and $J_{r, \mu}(t_{0} v_{0})<0$, then (ii) holds.
\end{proof}
As a consequence, we define the minimax value
$$
c_{r, \mu}:=\inf _{\gamma \in \Gamma_{r, \mu}} \max _{t \in[0,1]} J_{r, \mu}(\gamma(t))>0,
$$
where
$$
\Gamma_{r, \mu} := \{\gamma \in C([0,1],\, Y^{+} \cap U_{\mu}): \gamma(0)=0,\, J_{r, \mu}(\gamma(1))<0\}.
$$
Note that the functional $J_{r, \mu}$ is defined on the bounded set
$$
\{v \in Y^{+}:(T_{\mu} v, v)_{2}<1\} \subset\{v \in Y^{+}:\|v\|<\mu^{-\frac{1}{2}}\}.
$$
Similarly, $I_{r, \mu}$ is also defined on a bounded domain $U_{\mu}$. If $r_{1} \leq r_{2}$ and $\mu_{1} \leq \mu_{2}$, we have $c_{r_{1}, \mu_{2}} \leq c_{r_{2}, \mu_{1}}$ because $J_{r_{1}, \mu_{2}} \leq J_{r_{2}, \mu_{1}}$ on $U_{\mu_{2}}$.

Define $c_{\infty}:=\sup _{r>1, \mu>0} c_{r, \mu}$. For all $v \in Y^{+} \backslash\{0\}$, we have
\begin{equation}\label{eqcssup}
    c_{\infty} \leq \sup _{0 \leq t<|v|_{2}^{-1}} J(t v).
\end{equation}
We now obtain an important upper bound estimates for $c_\infty$. From now on, for simplicity, we will denote by $C$ or $C_i$(where $i=1,2,...)$ generic positive constants, whose the exact value may vary from line to line or even within the same line.

\begin{lemma}\label{lemcin}
    Let $2<p<4$. Then $c_\infty < \frac{m}{2}$.
\end{lemma}
\begin{proof}
Let the half-lines of the graph $\G$ be denoted by $\h_1, \h_2,...,\h_N$, and let the total length of the graph $\K$ be $\abs{\K} = \sum_{e \in \K}\ell_e$. Let $b > 0$. Define $\varphi_{b}^1: \G \to \mathbb{C}$ as
$$
    \varphi_{b}^1(x)= \begin{cases}1 & \text { for } x \in \K\\ \max\{0,1-bx\} & \text { for } x \in \h_i,\quad i = 1,2,...,N, \end{cases}
$$
and define $\varphi_{b}: \G \to \mathbb{C}^2$ by
$$\varphi_{b} = \binom{\varphi_{b}^1}{0}.$$

We first estimate  $J(\frac{\varphi_b^+}{\norm{\varphi_b^+}_2})$.
It is clear that $\varphi_{b} \in \operatorname{dom}(\D)$, $(\D- mI)\varphi_{b} = \left(0, -\imath \left(\varphi_{b}^1\right)^\prime\right)^T$ and $((\D- mI)\varphi_{b}, \varphi_{b})_2 = 0$. Thus, we obtain
\[
\norm{\left(\varphi_{b}^1\right)^\prime}_{L^2(\G)}\norm{\varphi_{b}^-}_2\geq \abs{((\D- mI)\varphi_{b},\varphi_{b}^-)_2} = \abs{((\D- mI)\varphi_{b}^-,\varphi_{b}^-)_2} = \norm{\varphi_{b}^-}^2 + m \norm{\varphi_{b}^-}^2_2 \geq 2m\norm{\varphi_{b}^-}_2^2.
\]
Hence,
\begin{equation}\label{eqphi-}
\norm{\varphi_{b}^-}^2 + m \norm{\varphi_{b}^-}^2_2 \leq \abs{((\D- mI)\varphi_{b},\varphi_{b}^-)_2} \leq \frac{1}{2m}\norm{\left(\varphi_{b}^1\right)^\prime}_{L^2(\G)}^2 = \frac{Nb}{2m}.
\end{equation}
Since $((\D- mI)\varphi_{b},\varphi_{b}^+)_2 + ((\D- mI)\varphi_{b},\varphi_{b}^-)_2 = ((\D- mI)\varphi_{b}, \varphi_{b})_2 = 0$, we obtain
\begin{equation}\label{eqphi+}
\norm{\varphi_{b}^+}^2 - m \norm{\varphi_{b}^+}^2_2 = ((\D- mI)\varphi_{b}^+,\varphi_{b}^+)_2 = ((\D- mI)\varphi_{b},\varphi_{b}^+)_2 \leq \frac{Nb}{2m}.
\end{equation}
Since $\norm{\varphi_b}_2^2 = \frac{N}{3b} + \abs{\K}$, \eqref{eqphi-} implies
\begin{equation}\label{eqphi+2}
  \frac{\norm{\varphi_b^+}^2_2}{\norm{\varphi_b}^2_2} \to 1, \quad \frac{\norm{\varphi_b^+}^2}{\norm{\varphi_b}^2} \to 1 \text{ and } \int_{\K}\abs{\varphi_b^+}^{p} \to \int_{\K}\abs{\varphi_b}^{p} = \abs{\K} \text{ as } b \to 0.
\end{equation}

Next, for $t > 0$, from the definition of $h$, we have
$$
I_0(t\varphi_b+h(t\varphi_b)) \geq I_0(t\varphi_b),
$$
which implies
\begin{equation}\label{eqhphi}
    \frac{p}{2}\norm{h(t\varphi_b)}^{2} \leq a\int_{\K}\abs{t\varphi_b+h(t\varphi_b)}^{p}\,dx + \frac{p}{2}\norm{h(t\varphi_b)}^{2} \leq a\int_{\K}\abs{t\varphi_b}^{p} \,dx = a t^p\abs{\K}.
\end{equation}

Thus, for sufficiently small $b>0$, using \eqref{eqphi+}-\eqref{eqhphi}, $p \in (2,4)$, Gagliardo-Nirenberg-Sobolev inequality in Lemma \ref{lemgnsg} and convexity of $\Psi$, we obtain
    \begin{equation}\label{eqjsmallm}
            \begin{aligned}
                J(\frac{\varphi_b^+}{\norm{\varphi_b^+}_2}) =& \frac{\norm{\varphi_b^+}^2}{2\norm{\varphi_b^+}^2_2}-\frac{1}{2}\left\lVert h\left(\frac{\varphi_b^+}{\norm{\varphi_b^+}_2}\right) \right\rVert^2 - \Psi\left(\frac{\varphi_b^+}{\norm{\varphi_b^+}_2}+h\left(\frac{\varphi_b^+}{\norm{\varphi_b^+}_2}\right)\right)\\
                \leq& \frac{m}{2} + \frac{\norm{\varphi_b^+}^2 - m\norm{\varphi_b^+}^2_2}{2\norm{\varphi_b^+}^2_2}   -\frac{1}{2}\left\lVert h\left(\frac{\varphi_b^+}{\norm{\varphi_b^+}_2}\right) \right\rVert^2\\
                & - 2^{1-p}\norm{\varphi_b^+}_2^{-p}\Psi(\varphi_b^+) +  \Psi\left(h\left(\frac{\varphi_b^+}{\norm{\varphi_b^+}_2}\right)\right)\\
                \leq& \frac{m}{2} + C_1b^2 - C_2ab^\frac{p}{2} - \frac{1}{2}\left\lVert h\left(\frac{\varphi_b^+}{\norm{\varphi_b^+}_2}\right)\right\rVert^2\left(1 - \frac{aC_{p,\K}}{pm}\left\lVert h\left(\frac{\varphi_b^+}{\norm{\varphi_b^+}_2}\right) \right\rVert^{p-2}\right) \\
                \leq& \frac{m}{2} + C_1b^2 - C_2ab^\frac{p}{2} - \frac{1}{2}\left\lVert h\left(\frac{\varphi_b^+}{\norm{\varphi_b^+}_2}\right)\right\rVert^2(1 - C_3 a^\frac{p}{2}b^\frac{p(p-2)}{4})\\
                <& \frac{m}{2}.
    \end{aligned}
    \end{equation}

Define $g(t):=J(t\varphi_b^+)$. Then, from \eqref{eqhphi} and the definition of $J$, for $0 \leq t \leq\norm{\varphi_b^+}_{2} ^{-1}$ and sufficiently small $b$, we have
$$
\begin{aligned}
g'(t) & =\frac{1}{t}\langle J^{\prime}(t\varphi_b^+), t\varphi_b^+\rangle=\frac{1}{t}\langle I_0'(t\varphi_b^++h(t\varphi_b^+)), t \varphi_b^++h'(t\varphi_b^+) t\varphi_b^+\rangle \\
& =\frac{1}{t}\langle I_0^{\prime}(t\varphi_b^++h(t\varphi_b^+)), t\varphi_b^+ + h(t\varphi_b^+)\rangle \\
& =\frac{1}{t}\left[\norm{t\varphi_b^+}^2-\norm{h(t\varphi_b^+)}^2-a\int_{\K} \abs{t\varphi_b^+ + h(t\varphi_b^+)}^p \,dx\right]\\
& \geq \frac{1}{t}\left[\norm{t\varphi_b^+}^{2}-a\int_{\K}\abs{t\varphi_b^+}^{p}\,dx+\frac{p-2}{2}\norm{h(t\varphi_b^+)}^{2}\right] \\
& \geq t\norm{\varphi_b^+}^2-at^{p-1}\int_{\K} \abs{\varphi_b^+}^{p} \,dx \\
& \geq t\left(\norm{\varphi_b^+}^2-a\norm{\varphi_b^+}_{2}^{2-p}\int_{\K} \abs{\varphi_b^+}^{p} \,dx\right)\\
& \geq t(C_1b^{-1} - C_2ab^\frac{p-2}{2}),
\end{aligned}
$$
which implies $g(t)$ is increasing for $0 \leq t \leq\norm{\varphi_b^+}_{2} ^{-1}$ when $b$ is sufficiently small. Now, fix $b$ sufficiently small. Then, by \eqref{eqcssup} and \eqref{eqjsmallm}, we have
$$
c_{\infty} \leq \sup _{0 \leq t \leq\norm{\varphi_b^+}_{2} ^{-1}} J(t\varphi_b^+) \leq J(\frac{\varphi_b^+}{\norm{\varphi_b^+}_2}) < \frac{m}{2}.
$$
This completes the proof of lemma.
\end{proof}
The mountain pass geometry allows us to find a Palais-Smale sequence $\{v_{n}\}$ satisfying
\begin{equation*}
J_{r, \mu}(v_{n}) \rightarrow c_{r, \mu}, \quad J_{r, \mu}^{\prime}(v_{n}) \rightarrow 0
\end{equation*}
which we call a mountain pass critical sequence. Define $u_{n}:=v_{n}+h_{r, \mu}(v_{n})$, $\omega_{n}=2f_{r}^{\prime}((T_{\mu} u_{n}, u_{n})_{2})$. We have the following Lemma.
\begin{lemma}\label{lemwn}
If $c_\infty < \frac{m}{2}$, then, for sufficiently large $r$, we have
$$
\limsup _{n \rightarrow \infty} \omega_{n} \leq 2 c_{\infty}.
$$
\end{lemma}
\begin{proof}
 First, observe that $f_{r}^{\prime}(s)s-f_{r}(s)$ and $f_{r}^{\prime}(s)$ are strictly increasing with respect to $s \in (0,1)$. Since $0<c_\infty < \frac{m}{2}$, by continuity of $s \mapsto f_{r}^{\prime}(s) s-f_{r}(s)$ and $s \mapsto f_{r}^{\prime}(s)$, there exist $\xi_{r}, \tau_{r} \in (0,1)$ such that
$$
c_{\infty}=f_{r}^{\prime}(\xi_{r})\xi_{r}-f_{r}(\xi_{r})\,\, \text{ and }\,\, 2f_{r}^{\prime}(\tau_{r})=m.
$$
Let $d_{r}=f_{r}^{\prime}(\tau_{r}) \tau_{r}-f_{r}(\tau_{r})$. By the properties of $f_{r}$, it follows that
$$
\xi_{r}, \tau_{r} \to 1, \quad d_{r} \to \frac{m}{2},\,\,\text{as}\,\,r \to \infty.
$$
This yields that $c_{\infty}<d_{r}$ for $r$ large enough, and implying $\xi_{r}<\tau_{r}$. Moreover, for $r$ large enough, by a simple computation, we have
$$
\begin{aligned}
d_{r}-c_{\infty} & =f_{r}^{\prime}(\tau_{r}) \tau_{r}-f_{r}(\tau_{r})-f_{r}^{\prime}(\xi_{r}) \xi_{r}+f_{r}(\xi_{r}) \\
& =\int_{\xi_{r}}^{\tau_{r}} f_{r}^{''}(s) s \,ds \\
& \leq \int_{\xi_{r}}^{\tau_{r}} f_{r}^{''}(s) \,ds=\frac{m}{2}-f_{r}^{\prime}(\xi_{r}).
\end{aligned}
$$
Thus, we have
\begin{equation}\label{eqsupfr}
\limsup_{r \to \infty} f_{r}^{\prime}(\xi_{r}) \leq c_{\infty}.
\end{equation}
Next, from the definition of $\{v_{n}\}$, we have
$$
\begin{aligned}
 o_{n}(1)\norm{v_{n}} &= \langle J_{r, \mu}^{\prime}(v_{n}), v_{n}\rangle=\langle I_{r, \mu}^{\prime}(v_{n}+h_{r, \mu}(v_{n})),v_{n}+h_{r, \mu}^{\prime}(v_{n}) v_{n})\rangle \\
& =\langle I_{r, \mu}^{\prime}(v_{n}+h_{r, \mu}(v_{n})),v_{n}+h_{r, \mu}(v_{n})\rangle \\
& =2 J_{r, \mu}(u_{n})-\frac{a(p-2)}{p} \int_{\K}\abs{u_{n}}^{p}\,dx+2 f_{r}((T_{\mu} u_{n}, u_{n})_{2})-2f_{r}^{\prime}((T_{\mu} u_{n}, u_{n})_{2})(T_{\mu} u_{n}, u_{n})_{2}.
\end{aligned}
$$
Moreover, there exists $\sigma_{r, \mu} \in(0,1)$ such that $c_{r, \mu}=f_{r}^{\prime}(\sigma_{r, \mu}) \sigma_{r, \mu}-f_{r}(\sigma_{r, \mu}).$ Therefore, we get
$$
\begin{aligned}
& \limsup_{n \to \infty}\left[f_{r}^{\prime}((T_{\mu} u_{n}, u_{n})_{2})(T_{\mu} u_{n}, u_{n})_{2}-f_{r}((T_{\mu} u_{n}, u_{n})_{2})\right] \leq \lim_{n \to \infty} J_{r, \mu}(u_{n}) \\
& =c_{r, \mu}=f_{r}^{\prime}(\sigma_{r, \mu}) \sigma_{r, \mu}-f_{r}(\sigma_{r, \mu}) \leq c_{\infty}=f_{r}^{\prime}(\xi_{r}) \xi_{r}-f_{r}(\xi_{r}).
\end{aligned}
$$
Since  $s \mapsto f_{r}(s)$ and  $s \mapsto f_{r}^{\prime}(s) s-f_{r}(s)$ are the strictly monotonic in $s$, we deduce that
$$
\limsup _{n \rightarrow \infty} \omega_{n}=2 \limsup _{n \rightarrow \infty} f_{r}^{\prime}\left(\left(T_{\mu} u_{n}, u_{n}\right)_{2}\right) \leq 2 f_{r}^{\prime}\left(\xi_{r}\right)
$$
and by \eqref{eqsupfr},
$$
\limsup _{n \rightarrow \infty} \omega_{n}=2 \limsup _{n \rightarrow \infty} f_{r}^{\prime}\left(\left(T_{\mu} u_{n}, u_{n}\right)_{2}\right) \leq 2 c_{\infty}.
$$
This completes the proof.
\end{proof}
\section{Proof of Theorem \ref{th1}}\label{sect3}
In this section, we provide the proof of Theorem \ref{th1}. To begin, we first establish the following result.
\begin{theorem}\label{th31}
If $c_\infty<\frac{m}{2}$, then, one of the following two cases holds:
\begin{enumerate}[label=\rm(\roman*)]
\item either there exist a non-trivial $u_{0} \in \operatorname{dom}(\D)$ and $\omega_{0} \in[0, m)$ such that
$$
\left\{\begin{array}{l}
\D u_{0}-\omega_{0} u_{0}=a\chi_{\K}\abs{u_{0}}^{p-2} u_{0},\\
\int_{\G}\left|u_{0}\right|^{2} \,dx = 1,
\end{array}\right.
$$
and
$$
0<\frac{1}{2}(\D u_{0}, u_{0})_{2}-\Psi(u_{0})=c_{\infty},
$$
\item or there exists a non-trivial $u_{0} \in \operatorname{dom}(\D)$ such that
$$
\left\{\begin{array}{l}
\D u_{0}=a\chi_{\K}\abs{u_{0}}^{p-2} u_{0}, \\
\int_{\G}\left|u_{0}\right|^{2} \,dx < 1,
\end{array}\right.
$$
and
$$
0<\frac{1}{2}(\D u_{0}, u_{0})_{2}-\Psi(u_{0})=c_{\infty}.
$$
\end{enumerate}
\end{theorem}
\begin{proof}
We  prove Theorem \ref{th31} in two steps.

{\bf Step 1} Fixed $r<\infty$ large and $\mu>0$ small.

From Lemma \ref{lemwn}, we may assume that $\omega_{n} \to \omega_{r, \mu} \in[0, c_{\infty}+\frac{m}{2}] \subset [0, m)$. By the boundedness of $\{u_{n}\}$ that, up to a subsequence, $u_{n} \rightharpoonup u_{r, \mu}$ in $Y$.  We claim that $u_{n} \to u_{r, \mu}$ in $Y$.

 Note that, $u_{n}$ and $u_{r, \mu}$ satisfy $I_{r,\mu}'(u_n) \to 0$ and $I_{0}'(u_{r, \mu}) - \omega_{r, \mu}T_{\mu}u_{r, \mu} = 0$, respectively. Taking the scalar product of $I_{r,\mu}'(u_n)$ and $I_{r,\mu}'(u_{r,\mu})$ with $u_{n}^{+}-u_{r, \mu}^{+}$, we obtain
\begin{equation}\label{equn+}
    \begin{aligned}
\norm{u_{n}^{+}-u_{r, \mu}^{+}}^{2}= & a\int_{\K} (\abs{u_{n}}^{p-2} u_{n}-\abs{u_{r, \mu}}^{p-2} u_{r, \mu})(u_{n}^{+}-u_{r, \mu}^{+}) \,dx  \\
& +(\omega_{n} T_{\mu} u_{n}-\omega_{r, \mu} T_{\mu} u_{r, \mu}, u_{n}^{+}-u_{r, \mu}^{+})_{2}+o_{n}(1).
\end{aligned}
\end{equation}
Since the embedding $ Y \hookrightarrow L^p(\K,\mathbb{C}^2)$ is compact, we can verify that
\begin{equation}
    \int_{\K} (\abs{u_{n}}^{p-2} u_{n}-\abs{u_{r, \mu}}^{p-2} u_{r, \mu})(u_{n}^{+}-u_{r, \mu}^{+}) \, dx \to 0,\,\,\text{as}\,\, n \rightarrow \infty.
\end{equation}
Moreover, for $\mu>0$ small, we have
\begin{equation}\label{eqomegantmu}
   \begin{aligned}
&(\omega_{n} T_{\mu} u_{n}-\omega_{r, \mu} T_{\mu} u_{r, \mu}, u_{n}^{+}-u_{r, \mu}^{+})_{2} \\
&\quad =(\omega_{n}-\omega_{r, \mu})(T_{\mu} u_{n}, u_{n}^{+}-u_{r, \mu}^{+})_{2}+\omega_{r, \mu}(T_{\mu}(u_{n}-u_{r, \mu}), u_{n}^{+}-u_{r, \mu}^{+})_{2} \\
&\quad =o_{n}(1)+\omega_{r, \mu}\norm{u_{n}^{+}-u_{r, \mu}^{+}}_{2}^{2}+\mu \omega_{r, \mu}\norm{u_{n}^{+}-u_{r, \mu}^{+}}^{2}\\
&\quad \leq o_{n}(1)+\frac{2 c_{\infty}+m}{2 m}\norm{u_{n}^{+}-u_{r, \mu}^{+}}^{2}+\frac{m-2 c_{\infty}}{4 m}\norm{u_{n}^{+}-u_{r, \mu}^{+}}^{2} \\
&\quad \leq o_{n}(1)+\frac{2 c_{\infty}+3 m}{4 m}\norm{u_{n}^{+}-u_{r, \mu}^{+}}^{2} .
\end{aligned}
\end{equation}
Thus, it follows from \eqref{equn+} to \eqref{eqomegantmu} that $u_{n}^{+} \to u_{r, \mu}^{+}$ in $Y$. Similarly, we can show $h_{r,\mu}(u_n^+) = u_{n}^{-} \to u_{r, \mu}^{-}$ in $Y$, and hence $u_{n} \to u_{r, \mu}$ in $Y$. Since $h \in C^1(Y^+, Y^-)$, we obtain that $h_{r,\mu}(u_{r,\mu}^+) = u_{r, \mu}^{-}$. Therefore, by continuity, we have
$$
J_{r, \mu}^{\prime}(v_{r, \mu})=0, \quad J_{r, \mu}(v_{r, \mu})=c_{r, \mu}, \quad (T_{\mu} u_{r, \mu}, u_{r, \mu})_{2}<1,
$$
$$
\omega_{r, \mu}=2 f_{r}^{\prime}((T_{\mu} u_{r, \mu}, u_{r, \mu})_{2}) \in[0, c_{\infty}+\frac{m}{2}] \subset[0, m),
$$
and
$$
0<
\frac{1}{2}\norm{u_{r, \mu}^+}^2 - \frac{1}{2}\norm{u_{r, \mu}^-}^2-\Psi(u_{r, \mu})=c_{r, \mu}+H_{r, \mu}(u_{r, \mu})<c_{r, \mu}+\frac{\omega_{r, \mu}}{2r}.
$$
{\bf Step 2} As $r_{n} \to \infty$ and $\mu_{n} \to 0$.

Let $r_{n}>1$ and $\mu_{n}>0$ be such that $r_{n} \nearrow \infty$ and $\mu_{n} \searrow 0$ (monotone sequences).  Consider a sequence $\left\{v_{r_{n}, \mu_{n}}\right\}$ such that
$$
J_{r_{n}, \mu_{n}}\left(v_{r_{n}, \mu_{n}}\right)=c_{r_{n}, \mu_{n}}, \quad J_{r_{n}, \mu_{n}}^{\prime}\left(v_{r_{n}, \mu_{n}}\right)=0.
$$
Since the domain of $J_{r_{n}, \mu_{n}}$ depends on $n$. Similar to step 1, define $\tilde{u}_n := $ $v_{r_{n}, \mu_{n}}+h_{r_{n}, \mu_{n}}\left(v_{r_{n}, \mu_{n}}\right)$ and $\tilde{\omega}_{n}:=2 f_{r_{n}}^{\prime}((T_{\mu_{n}} u_{r_{n}, \mu_{n}}, u_{r_{n},\mu_{n}})_{2})$. We have $I_{r_{n}, \mu_{n}}(\tilde{u}_n)=c_{r_n, \mu_n}$, $I_{r_{n}, \mu_{n}}^{\prime}(\tilde{u}_n)=0$ and $\tilde{\omega}_{n} \leq c_{\infty}+\frac{m}{2}<m$.
Then, since $f_r^{\prime}(s)s > rf_r(s)$ for $0 < s < 1$, we obtain
$$
\begin{aligned}
2c_\infty + o_n(1) &= 2I_{r_{n}, \mu_{n}}(\tilde{u}_n) - \langle I_{r_{n}, \mu_{n}}^{\prime}(\tilde{u}_n),\tilde{u}_n\rangle  \\
&= \frac{a(p-2)}{p}\int_{\K}\abs{\tilde{u}_n}^p\, dx + 2f_{r_n}^{\prime}((T_{\mu_n} \tilde{u}_n,\tilde{u}_n)_2)(T_{\mu_n} \tilde{u}_n, \tilde{u}_n)_2 - 2f_{r_n}((T_{\mu_n} \tilde{u}_n, \tilde{u}_n)_2)\\
& \geq \frac{a(p-2)}{p}\int_{\K}\abs{\tilde{u}_n}^p\, dx.
\end{aligned}
$$
Hence, by Sobolev inequality, $\langle I_{r_{n}, \mu_{n}}^{\prime}(\tilde{u}_n),\tilde{u}_n^+ \rangle = 0$ implies
\[
\begin{aligned}
        \norm{\tilde{u}_n^+}^2 &\leq a\int_{\K}\abs{\tilde{u}_n}^{p-1}\abs{\tilde{u}_n^+}\, dx + \tilde{\omega}_n\norm{\tilde{u}_n^+}_2^2 + \mu_n\tilde{\omega}_n\norm{\tilde{u}_n^+}^2\\
        & \leq a\norm{\tilde{u}_n}_p^{p-1}\norm{\tilde{u}_n^+}_p + \frac{2c_{\infty}+m}{2m}(1+\mu_nm)\norm{\tilde{u}_n^+}^2 \\
        &\leq C\norm{\tilde{u}_n^+} + \frac{2c_{\infty}+m}{2m}(1+\mu_nm)\norm{\tilde{u}_n^+}^2.
\end{aligned}
\]
It follows that $\{\tilde{u}_n^+\}$ is bounded in $Y$. Moreover, since
\[
\frac{1}{2}\norm{\tilde{u}_n^+}^2 - \frac{1}{2}\norm{\tilde{u}_n^-}^2 > J_{r_n,\mu_n}(\tilde{u}_n) = c_{r_n,\mu_n} > 0,
\]
we know that $\{\tilde{u}_n\}$ is bounded in $Y$.
Thus, repeating the arguments of step 1, we may assume that $\tilde{u}_n \to u_{0}$ in $Y$ and $\tilde{\omega}_n \rightarrow \omega_{0}$ as $n \rightarrow \infty$, and satisfying
$I_{\omega_0}^{\prime}(u_0) = 0$, $\norm{u_{0}}_{2} \leq 1$, $0 \leq \omega_{0} \leq c_{\infty}+\frac{m}{2} < m$, and
$$
0<c_{r_{1}, \mu_{1}} \leq \frac{1}{2}\norm{u_0^+}^2 - \frac{1}{2}\norm{u_0^-}^2-\Psi(u_{0})=c_{\infty}.
$$
Now, either $\norm{u_{0}}_{2}=1$ or $\norm{u_{0}}_{2}<1$. When the latter happens,
$$
0 \leq \omega_{0}=\lim _{n \rightarrow \infty} 2f_{r_{n}}^{\prime}(\norm{\tilde{u}_n}_{2}^{2}) \leq \limsup _{n \rightarrow \infty} 2f_{r_{n}}^{\prime}\left(\frac{1+\norm{u_{0}}_{2}^{2}}{2}\right)=0,
$$
that is $\omega_{0}=0$. By \cite[Proposition 3.1]{Bo}, we know that $u \in \operatorname{dom}(\D)$. This completes the proof.
\end{proof}
Now, we will provide the proof of Theorem \ref{th1}.
\begin{proof}[Proof of Theorem \ref{th1}] Apply Lemma \ref{lemcin} and Theorem \ref{th31} to the operator $\D-sI$ for any $s \in (-m, m)$ instead of $\D$. From the proof of Theorem \ref{th31}, one of the following two cases must hold:
    \begin{enumerate}[label=\rm(\roman*)]
\item either there exist a non-trivial $u_{s} \in \operatorname{dom}(\D)$ and $\omega_{s} \in[0, m-s)$ (since $\sigma(\D- sI)=(-\infty,-m-s]\cup[m-s, \infty)$) such that
$$
\left\{\begin{array}{l}
\D u_s-su_{s}-\omega_{s} u_{s}=a\chi_{\K}\abs{u_{s}}^{p-2} u_{s}, \\
\int_{\K}\abs{u_{s}}^{2} \,dx=1,
\end{array}\right.
$$
and
$$
0<\frac{1}{2}((\D-s) u_{s}, u_{s})_{2}-\Psi(u_{s}) = c_{s,\infty}:=\sup _{ r>1, \mu>0} c_{s, r, \mu},
$$
where
$$c_{s, r, \mu} := \inf _{\gamma \in \Gamma_{s, r, \mu}} \max _{t \in[0,1]} J_{s, r, \mu}(\gamma(t))>0,$$
$$
\Gamma_{s, r, \mu} :=\{\gamma \in C([0,1], Y^{+} \cap U_{\mu}): \gamma(0)=0, J_{s, r, \mu}(\gamma(1))<0\},
$$
with $J_{s, r, \mu}: Y^+ \cap U_\mu \to \mathbb{R}$ defined by
$$
J_{s, r, \mu}(v):=\max \left\{I_{r, \mu}(v+w) -s \int_{\G}\abs{v+w}^2\, dx: w \in Y^{-},(T_{\mu}(v+w), v+w)_{2}<1\right\}.
$$
\item or there exists a non-trivial $u_{s} \in \operatorname{dom}(\D)$ such that
$$
\left\{\begin{array}{l}
\D u_s-su_{s}=a\chi_{\K}\abs{u_{s}}^{p-2} u_{s}, \\
\int_{\K}\abs{u_{s}}^{2} \,dx<1,
\end{array}\right.
$$
and
$$
0<\frac{1}{2}((\D-s) u_{s}, u_{s})_{2}-\Psi(u_{s}) = c_{s,\infty}.
$$
\end{enumerate}
Since $s+\omega_{s} \in(-m, m)$, and the proof of Theorem \ref{th1} is complete.
\end{proof}
\section{Proof of Theorem \ref{th2}}\label{sect4}
Assume that the second case of Theorem \ref{th1} occurs. Let $u_{0}$ be a non-trivial solution corresponding to $s =  0$. Then $u_{ 0}$ satisfies
$$
\left\{\begin{array}{l}
\D u_{ 0}=a \chi_{\K}\abs{u_{0}}^{p-2} u_{0}, \\
\int_{\G}\abs{u_{0}}^{2} \,dx<1,
\end{array}\right.
$$
which implies $I_{0}^{\prime}(u_0)=0$ and $\norm{u_0}^2<1$.

It follows from $2<p<4$, Lemma \ref{lemgnsg} and the H\"older inequality that, for all $u \in Y$,
\begin{equation}\label{eqho2}
\begin{aligned}
        \int_{\K}\abs{u}^p\, dx &\leq C_{p,\K}\norm{u}^{p-2}\norm{u}_2^2 \\
    &\leq C_{p,\K}\norm{u}^{p-2}\norm{u}_2^{4-p}\norm{u}_2^{p-2}\\
    &\leq \frac{C_{p,\K}}{m^\frac{4-p}{2}}\norm{u}^2\norm{u}_2^{p-2}.
\end{aligned}
\end{equation}
Then, using the fact that $\langle I_{0}^{\prime}(u_0), u^+_0 - u^-_0\rangle = 0$, and applying \eqref{eqho2} and the H\"older inequality, we obtain the following
\begin{equation}\label{eqi+-}
    \begin{aligned}
\norm{u_{0}}^{2} & = a\int_{\K}\abs{u_{0}}^{p-2} u_{0}(u_{0}^{+}-u_{0}^{-}) \,dx\\
& \leq a \int_{\K}\abs{u_{0}}^{p-1}(\abs{u_{0}^{+}}+\abs{u_{0}^{-}}) \,dx \\
& \leq a (\int_{\K}\abs{u_{0}}^p\,dx)^\frac{p-1}{p}\left[(\int_{\K}\abs{u_{0}^{+}}^p \,dx)^\frac{1}{p} + (\int_{\K}\abs{u_{0}^{-}}^p \,dx)^\frac{1}{p}\right]\\
& \leq \frac{C_{p,\K}a}{m^\frac{4-p}{2}}(\norm{u_0}^2\norm{u_0}_2^{p-2})^\frac{p-1}{p}\left[(\norm{u^+_0}^2\norm{u^+_0}_2^{p-2})^\frac{1}{p} + (\norm{u^-_0}^2\norm{u^-_0}_2^{p-2})^\frac{1}{p}\right]\\
& \leq \frac{2C_{p,\K}a}{m^\frac{4-p}{2}}(\norm{u_0}^2\norm{u_0}_2^{p-2})^\frac{p-1}{p}(\norm{u_0}^2\norm{u_0}_2^{p-2})^\frac{1}{p}\\
& \leq \frac{2C_{p,\K}a}{m^\frac{4-p}{2}}\norm{u_0}^2,
\end{aligned}
\end{equation}
which leads to a contradiction  when $a < \frac{m^\frac{4-p}{2}}{2C_{p,\K}} = a_0$. Thus, the proof of Theorem \ref{th2} is complete.
\section{Proof of Theorem \ref{thalla}}\label{sectalla}
As in \cite{Bu}, we can show that, if for each $s \in (-m,m)$, the second case in the proof of Theorem \ref{th1} occurs, then there exists non-trivial $u_{-m}\in \operatorname{dom}(\D)$ such that
\[
\D u_{-m}+mu_{-m}=a\chi_{\K}\abs{u_{-m}}^{p-2} u_{-m}.
\]
However, since the corresponding spectrum results in \cite{Bu} are missing for Dirac operators, we will need Lemma \ref{lemtree} to exclude the existence of $u_{-m}$. This result depends on the topology of $\G$ and its proof utilizes techniques from graph theory.

To prove the main result in this section, we  first introduce some basic definitions for unordered combinatorial graphs, see \cite{Gr,Be,Bon} for more details. In what follows, we will also regard $\K$ as a unordered combinatorial graph.
\begin{definition}\label{defcomg}
A unordered combinatorial graph $G$ is a couple $(V,E)$, where $V$ is a set of vertices and $E$ is a set of edges, that is, $E$ consists of some unordered couples $vw$, where $v,w\in V$.

  Let $G=(V,E)$ be a unordered combinatorial graph.
  We write $v \sim w$ ($v$ is
adjacent to $w$) if $vw \in E$. For each vertex $v\in V$, define its degree $\operatorname{deg}_G(v) = \#\{w \in V: w\sim v\}$.
  A fnite sequence $\{v_k\}_{k=0}^n$ of vertices is called a path if $v_k \sim v_{k+1}$ for all $k = 0,1,...,n$. The number $n$ of edges in the path is referred to as the length of the path. A simple path is a path in which all vertices are distinct. A closed path is a path $\{v_k\}_{k=0}^n$ such that $v_0=v_n$. A closed path is a simple cycle if it does not pass through the same edge or vertex more than once.

  A unordered combinatorial graph $G$ is called connected if, for any two vertices $v,w\in V$, there is a path connecting $v$ and $w$, that is, a path $\{v_k\}_{k=0}^n$ such that $v_0 = v$ and $v_n = w$. A unordered connected combinatorial graph without simple cycles is called a tree. A leaf of a tree is a vertex of degree $1$.

  If a unordered combinatorial graph $G$ is connected, then define the graph distance $d_G(v,w)$ between any two distinct vertices $v,w$ as follows: if $v \not = w$ then $d_G(v,w)$ is the minimal length of a path that connects $v$ and $w$, and if $v = w$ then $d_G(v,w)=0$.

\end{definition}
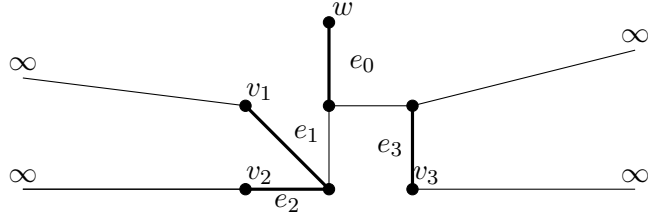
\begin{figure}[H]
	\begin{tikzpicture}[xscale= 0.37,yscale=0.37]
		\draw[very thick] (-14,6)--(-11,3) ; \draw (0,8)--(-8,6);
		\draw  (-11,6)--(-8,6); \draw (-22,3)--(-14,3);
        \draw [very thick] (-14,3)--(-11,3);
				\draw (0,3)--(-8,3);\draw (-22,3)--(-14,3);
		\draw [very thick] (-11,9)--(-11,6);\draw (-14,3)--(-11,3);
		\draw   (-11,3)--(-11,6);\draw (-14,6)--(-22,7);
		\draw [very thick] (-8,3)--(-8,6);
		\node at (-14,3) [nodo] {} ; \node at (-14,6) [nodo] {};
		\node at (-11,3) [nodo] {}; \node at (-11,6) [nodo] {} ;
		\node at (-8,3) [nodo] {}; \node at (-8,6) [nodo] {};
		\node at (-11,9) [nodo] {};
  \draw (-22,3.5)  node{$\infty$};  \draw (-22,7.5) node{$\infty$};
  	\draw (0,8.5)  node{$\infty$};
     	\draw (0,3.5)  node{$\infty$};
             	\draw (-10.5,9.5)  node{$w$};\draw (-9.8,7.5)  node{$e_0$};
                             	\draw (-13.5,6.5)  node{$v_1$};\draw (-11.8,5)  node{$e_1$};
             	\draw (-13.5,3.5)  node{$v_2$};\draw (-12.5,2.5)  node{$e_2$};
             	\draw (-7.5,3.5)  node{$v_3$};\draw (-8.8,4.5)  node{$e_3$};

		\end{tikzpicture}
	\caption{$\K$ is a tree with at most one leaf incident with no half-line in $\G$}
\end{figure}
\begin{lemma}\label{lemtree}
    Let $p \geq 2$. Assume that $u \in \operatorname{dom}(\D)$ such that
\[
\D u-\omega u =a\chi_{\K}\abs{u}^{p-2} u
\]
for some $\omega \in \mathbb{R}$. Then
    \begin{enumerate}[label=\rm(\roman*)]
        \item
for any vertex $v \in \mathrm{V}$ and edge $e\succ v$, if $u_f\equiv 0$ for all $f\succ v$ and $f\neq e$, then $u_e\equiv 0$.
\item if $\K$ is a tree with at most one leaf incident with no half-line in $\G$ and $u_{\mathcal{H}} \equiv 0$ on every half-line $\mathcal{H}\in\mathrm{E}$, then $u \equiv 0$.
\end{enumerate}
\end{lemma}
\begin{proof}(i)
By \cite[Remark 6 in Chapter 8]{Ha}, we obtain $u_e\in C^1(I_e,\mathbb{C}^2)$ for all $e \in \mathrm{E}$. Using the existence and uniqueness theorem (see, e.g., \cite[Section 2.2]{Pe}), we know that for any bounded edges $e \in \mathrm{E}$, if  $u_e(0)=0$ or $u_e(\ell_e)=0$, then $u_e\equiv 0$ on $e$. Hence (i) follows from \eqref{eqdefdom1} and \eqref{eqdefdom2}.

(ii) Since $\G$ is connected, it is easy to see that the set of vertices of $\K$ is exactly $\mathrm{V}$. Without loss of generality, we can assume that there exists exactly one leaf of $\K$ incident with no half-line in $\G$ and denote it by $v_0$. Denote the only bounded edge incedent with $v_0$ by $e_0$.
Let the other leaves of $\K$ be denoted by $v_1,v_2,...,v_N$ and the edge of $\K$ incident with $v_i$ be denoted by $e_i$ for $i=1,2,..,N$. Since $u_{\mathcal{H}} \equiv 0$ on every half-line $\mathcal{H}\in\mathrm{E}$, by (i),
we see $u_{e_i}\equiv 0$  for all $i=1,2,..,N$.

 Argue by contradiction, suppose that $S:=\{ v \in \mathrm{V}: u_e\not\equiv 0 \text{ for some }e\succ v\} \neq \emptyset$. It is clear that $v_i\not\in S$ for all $i=1,2,..,N$ and it follows $S\cap \{v\in\mathrm{V}:\operatorname{deg}_\K(v) = 1\} \subset \{v_0\}$. We show that $S\neq \{v_0\}$, otherwise $u_{e_0} \not \equiv 0$, but, since there exists $v \in \mathrm{V}$ such that $v\not \in S$ and $e_0\succ v$ in $\G$, we know $u_{e_0}\equiv 0$, which is a contradiction.  Thus, $S\backslash \{v\in\mathrm{V}:\operatorname{deg}_\K(v) = 1\} \neq \emptyset$ and it follows $\max_{v\in S}\operatorname{deg}_\K(v)\geq 2$. Then there exists an vertex $w \in S$ such that
 \begin{equation}\label{eqtree}
     d_\K(w,v_0)=\max_{v\in S}d_\K(v,v_0)=\max_{v\in S, \operatorname{deg}_\K(v)\geq2}d_\K(v,v_0) >0.
 \end{equation}
 Denote the vertices adjacent to $w$ by $w_1,w_2,..,w_d$, where $d:=\operatorname{deg}_\K(w)\geq2$. Since $\K$ is a tree, we know that any two vertices of $\K$ are connected by exactly one simple path in $\K$(see, e.g., \cite[Proposition 4.1]{Bon}). Denote the simple path which connects $v_0$ and $w$ by $P_{0}$. Without loss of generality, we can assume that $w_1$ lies on $P_0$. Thus,  $d_\K(w_i,v_0)= d_\K(w,v_0)+1$ for $i=2,...,d$. Using \eqref{eqtree}, we see that $w_i\not \in S$ for all $i=2,...,d$. Therefore, by(i), we obtain $w \not \in S$, which is a contradiction and this completes the proof.


\end{proof}
It is the position to provide the proof of Theorem \ref{thalla}.
\begin{proof}[Proof of Theorem \ref{thalla}]
Argue by contradiction, suppose that, for each $s \in (-m,m)$, the second case in the proof of Theorem \ref{th1} occurs, i.e., for any $s \in (-m,m)$, there exists a non-trivial $u_{s} \in \operatorname{dom}(\D)$ such that
$$
\left\{\begin{array}{l}
\D u_s-su_{s}=a\chi_{\K}\abs{u_{s}}^{p-2} u_{s}, \\
\int_{\K}\abs{u_{s}}^{2} \,dx<1,
\end{array}\right.
$$
and
\begin{equation}\label{eqenergys}
    0<\frac{1}{2}((\D-s) u_{s}, u_{s})_{2}-\Psi(u_{s}) = c_{s,\infty}.
\end{equation}
It is clear that, for any $s \in (-m,0)$, we have $ c_{s,\infty} \geq c_{0,\infty}= c_{\infty}>0$. Moreover, similar to the argument of Lemma \ref{lemcinsup}, we obtain that, for any $s \in (-m,m)$, $ c_{s,\infty}<\frac{m-s}{2}$. Thus, for any $s \in (-m,0)$,
\begin{equation}\label{eqcsinfty}
    m> \frac{m-s}{2} > c_{s,\infty} \geq c_{\infty}>0.
\end{equation}
It follows from $\langle I_{s}^{\prime}(u_{s}), u_{s}\rangle = 0$, \eqref{eqenergys} and \eqref{eqcsinfty} that
\begin{equation}\label{equsp}
    m\geq \frac{a(p-2)}{2p}\int_{\K}\abs{u_{s}}^p\, dx   \geq c_{\infty}>0.
\end{equation}
Since $\langle I_{{s}}^{\prime}(u_{s}), u^+_{s}\rangle = 0$, by Lemma \ref{lemgnsg}, \eqref{equsp} and the H\"older inequality, we have
\begin{equation*}
    \begin{aligned}
    \norm{u^+_{s}}^2  &\leq a\int_{\K}\abs{u_{{s}}}^{p-1}\abs{u_{{s}}^+}\, dx +\abs{s}\norm{u^+_{s}}^2_2 \\
    &\leq a(\int_{\K}\abs{u_{{s}}}^p\, dx)^\frac{p-1}{p}(\int_{\K}\abs{u_{{s}}^+}^p\, dx)^\frac{1}{p}+\abs{s} \\
    &\leq a\left[\frac{2mp }{a(p-2)}\right]^\frac{p-1}{p}\left(\frac{C_{p,\K}}{m}\right)^\frac{1}{p}\norm{u^+_{s}}+\abs{s}\\
    &= (C_{p,\K}a)^\frac{1}{p} m^\frac{p-2}{p}\left(\frac{2p }{p-2}\right)^\frac{p-1}{p}\norm{u^+_{s}}+\abs{s},
    \end{aligned}
\end{equation*}
and thus, by $\frac{1}{2}\norm{u^+_s}^2 - \frac{1}{2}\norm{u^-_s}^2 \geq I_s(u_s)+\frac{s}{2}\norm{u_{s}}^2_2 > -\frac{\abs{s}}{2}$, we know $\{u_s\}$ is bounded in $Y$. Let $s_n\to -m$ as $n\to \infty$. Then, it follows from \eqref{equsp} that, up to a subsequence, $u_{s_n} \rightharpoonup u_{-m}\not\equiv0$ in $Y$ and $u_{s_n} \to u_{-m}$ in $L^p(\K,\mathbb{C}^2)$. Since the embedding $ Y \hookrightarrow L^p(\K,\mathbb{C}^2)$ is compact, we can verify that, for any $v \in Y$,
\begin{equation}
    \int_{\K} (\abs{u_{s_n}}^{p-2} u_{s_n}v-\abs{u_{-m}}^{p-2} u_{-m}v) \, dx \to 0,\,\,\text{as}\,\, n \rightarrow \infty,
\end{equation}
and thus $I'_{-m}(u)=0$. By a similar argument as in  the
proof of proposition 3.1 in \cite{Bo}, we know $u_{-m}\in \operatorname{dom}(\D)$ solves the equation
\[
\D u_{-m}+mu_{-m}=a\chi_{\K}\abs{u_{-m}}^{p-2} u_{-m}.
\]
Moreover, by \cite[Remark 6 in Chapter 8]{Ha}, we get $u_e\in C^1(I_e,\mathbb{C}^2)$ for all $e \in \mathrm{E}$.
Using \eqref{eqdefde}, we see $u_{\mathcal{H}} \equiv 0$ on every half-line $\mathcal{H}\in\mathrm{E}$. It follows from Lemma \ref{lemtree}(ii) that $u\equiv 0$, which is a contradiction. Hence, for some $s \in (-m,m)$, the first case in the proof of Theorem \ref{th1} occurs, i.e., there exist a non-trivial $u_{s} \in \operatorname{dom}(\D)$ and $\omega_{s} \in[0, m-s)$ such that
$$
\left\{\begin{array}{l}
\D u_s-su_{s}-\omega_{s} u_{s}=a\chi_{\K}\abs{u_{s}}^{p-2} u_{s}, \\
\int_{\K}\abs{u_{s}}^{2} \,dx=1,
\end{array}\right.
$$
and
$$
0<\frac{1}{2}((\D-s) u_{s}, u_{s})_{2}-\Psi(u_{s}) = c_{s,\infty}.
$$
\end{proof}
\begin{remark}
Assume that $\K$ is a tree with at most one leaf incident with no half-line in $\G$. Define $Q:\G \to \mathbb{C}^2$ such that $Q\in C(I_e,\mathbb{C}^2)$ for all $e\in\mathrm{E}$ and $Q_{\mathcal{H}} \equiv 0$ on every half-line $\mathcal{H}\in\mathrm{E}$. Then, by applying a similar proof as above, one can show that $\pm m$ are not the eigenvalue of the operator $\D-Q$ defined on the domain $\operatorname{dom}(\D)$.
\end{remark}
\section{Proof of Theorem \ref{th3}}\label{sect5}
 Let $p \geq 4$. For $s \in [\omega_*,m)$, if $c_{s,\infty} < \frac{m-s}{2}$, then the proof of Theorem \ref{th3} follows exactly the same  reasoning as that of Theorem \ref{th1}  for $s\in [\omega_*,m)$. For brevity, we omit the details and proceed to complete the proof of Theorem \ref{th3} with the following essential estimate for $c_{s,\infty}$.
\begin{lemma}\label{lemcinsup}
    Let $p \geq 4$. Then, for all $\omega_* \in (-m,m)$, there exists $a^*_{\omega_*}$ such that, for all $a>a^*_{\omega_*}$, we have $c_{s,\infty} < \frac{m-s}{2}$ for all $s\in [\omega_*,m)$.
\end{lemma}
\begin{proof}
As in Section \ref{subsecper}, for $s\in (-m,m)$, define $J_s: Y^{+} \rightarrow \mathbb{R}$ by
$$
J_s(v)=I_s(v+h_s(v))=\max \{I_{s}(v+w): w \in Y^{-}\},
$$
where $h_s \in C^1\left(Y^{+}, Y^{-}\right)$, and the critical points of $J_s$ and $I_s$ are in one-to-one correspondence via the injective map $u \rightarrow u+h_s(u)$ from $Y^{+}$ into $Y$.

Let $t > 0$ and $s\in [\omega_*,m)$. Repeating the arguments of Lemma \ref{lemcin}, we obtain
\begin{equation}\label{eqhphisup}
\begin{aligned}
        \min\left\{1, 1 + \frac{\omega_*}{m}\right\}\frac{p}{2}\norm{h_s(t\varphi_b)}^{2} &\leq \min\left\{1, 1 + \frac{s}{m}\right\}\frac{p}{2}\norm{h_s(t\varphi_b)}^{2}\\
        &\leq a\int_{\K}\abs{t\varphi_b+h_s(t\varphi_b)}^{p}\,dx + \frac{p}{2}\norm{h_s(t\varphi_b)}^{2} + \frac{sp}{2}\norm{h_s(t\varphi_b)}^{2}_2 \\
        &\leq a\int_{\K}\abs{t\varphi_b}^{p} \,dx = a t^p\abs{\K}
\end{aligned}
\end{equation}
and, for sufficiently small $b$, one has
\begin{equation*}\label{eqjsmallmsup}
            \begin{aligned}
                J_{s}(\frac{\varphi_b^+}{\norm{\varphi_b^+}_2}) =& \frac{\norm{\varphi_b^+}^2}{2\norm{\varphi_b^+}^2_2}-\frac{1}{2}\left\lVert h_s\left(\frac{\varphi_b^+}{\norm{\varphi_b^+}_2}\right) \right\rVert^2 - \frac{s}{2} - \frac{s}{2}\left\lVert h_s\left(\frac{\varphi_b^+}{\norm{\varphi_b^+}_2}\right) \right\rVert_2^2\\
                &- \Psi\left(\frac{\varphi_b^+}{\norm{\varphi_b^+}_2}+h_s\left(\frac{\varphi_b^+}{\norm{\varphi_b^+}_2}\right)\right)\\
                \leq& \frac{m-s}{2}+ C_1b^2 - C_2ab^\frac{p}{2} \\&- \frac{1}{2}\left\lVert h_s\left(\frac{\varphi_b^+}{\norm{\varphi_b^+}_2}\right)\right\rVert^2\left(\min\left\{1, 1 + \frac{\omega_*}{m}\right\} - C_3 \left(\min\left\{1, 1 + \frac{\omega_*}{m}\right\}\right)^\frac{2-p}{2}a^\frac{p}{2}b^\frac{p(p-2)}{4}\right).
    \end{aligned}
    \end{equation*}
Define $g_s(t):=J_s(t\varphi_b^+)$. Then, from \eqref{eqhphisup} and the definition of $J_s$, for $0 \leq t \leq\norm{\varphi_b^+}_{2} ^{-1}$ and sufficiently small $b$, we have
$$
\begin{aligned}
g_s'(t)  =&\frac{1}{t}\langle J_s^{\prime}(t\varphi_b^+), t\varphi_b^+\rangle=\frac{1}{t}\langle I_s'(t\varphi_b^++h_s(t\varphi_b^+)), t \varphi_b^++h_s^\prime(t\varphi_b^+) t\varphi_b^+\rangle \\
 =&\frac{1}{t}\langle I_s^{\prime}(t\varphi_b^++h_s(t\varphi_b^+)), t\varphi_b^+ + h_s(t\varphi_b^+)\rangle \\
 =&\frac{1}{t}\left[\norm{t\varphi_b^+}^2-\norm{h(t\varphi_b^+)}^2-a\int_{\K} \abs{t\varphi_b^+ + h(t\varphi_b^+)}^p \,dx-s\norm{t\varphi_b^+}^2_2-s\norm{h(t\varphi_b^+)}^2_2\right]\\
 \geq& \frac{1}{t}\left[\norm{t\varphi_b^+}^{2} - s\norm{t\varphi_b^+}^2_2-a\int_{\K}\abs{t\varphi_b^+}^{p}\,dx+ \frac{p-2}{2}\norm{h(t\varphi_b^+)}^{2} + \frac{s(p-2)}{2}\norm{h(t\varphi_b^+)}^{2}_2\right] \\
 \geq& t\norm{\varphi_b^+}^2-at^{p-1}\int_{\K} \abs{\varphi_b^+}^{p} \,dx + \frac{(m+s)(p-2)}{2t}\norm{h(t\varphi_b^+)}^{2}_2\\
 \geq& t\left(\norm{\varphi_b^+}^2-a\norm{\varphi_b^+}_{2}^{2-p}\int_{\K} \abs{\varphi_b^+}^{p} \,dx\right)\\
 \geq& t(C_1b^{-1} - C_2ab^\frac{p-2}{2}).
\end{aligned}
$$
Let $b = a^\frac{2}{3-p}$. Then, since $p \geq 4$, there exists $a_{\omega_*} >0$ such that, for all $a > a_{\omega_*}$, we have $J_s(\frac{\varphi_b^+}{\norm{\varphi_b^+}_2}) < \frac{m-s}{2}$ and $g_s(t)$ is increasing for $0 \leq t \leq\norm{\varphi_b^+}_{2} ^{-1}$. Moreover,
$$
c_{s,\infty} \leq \sup _{0 \leq t \leq\norm{\varphi_b^+}_{2} ^{-1}} J_s(t\varphi_b^+) \leq J_s(\frac{\varphi_b^+}{\norm{\varphi_b^+}_2}) < \frac{m-s}{2}.
$$
This completes the proof of Lemma \ref{lemcinsup}.
\end{proof}
\section{Proof of Theorem \ref{th4}}\label{sect6}
 Let $4 \leq p <6$. In this section, we will replace $\Psi$ in Section \ref{subsecper} with $\Psi_\ell : Y \to \mathbb{R}$ given by
 \[
    \Psi_\ell(u) = \frac{a}{p}\int_{\K\cup\mathcal{H}_\ell}|u|^{p}\,dx.
 \]
 Again, for $s \in [\omega_*,m)$, if $c_{s,\infty} < \frac{m-s}{2}$, the proof of the first part of Theorem \ref{th4} follows exactly the same line as that of Theorem \ref{th1} for $s\in [\omega_*,m)$. For the sake of brevity, we omit the details and only show the essential estimate for $c_{s,\infty}$.
\begin{lemma}\label{lemcin46}
    Let $2 < p <6$. Then, for all $\omega_* \in (-m,m)$ and $a > 0$, there exists $\ell^*(\omega_*,a)$ such that, for all $\ell>\ell^*(\omega_*,a)$, we have $c_{s,\infty} < \frac{m-s}{2}$ for all $s\in [\omega_*,m)$.
\end{lemma}
\begin{proof}
    Let $\varphi^1 \in C_c^\infty(\mathcal{H})$. Define $\varphi_{b}^1: \G \to \mathbb{C}$ as
$$
    \varphi_{b}^1(x)= \begin{cases}\varphi^1(bx) & \text { for }  x\in\mathcal{H} \\0 & \text { for } x \in \G\backslash\mathcal{H}, \end{cases}
$$
and define $\varphi_{b}: \G \to \mathbb{C}^2$ by
$$\varphi_{b} = \binom{\varphi_{b}^1}{0},$$
where we tacitly identified the edge $\mathcal{H}$ with $[0, +\infty)$.
Then, we have $\norm{\left(\varphi_{b}^1\right)^\prime}_{L^2(\G)}^2 = b\norm{\left(\varphi^1\right)^\prime}_{L^2(\mathcal{H})}^2$, $\norm{\varphi_{b}}_2^2 = b^{-1}\norm{\varphi^1}_{L^2(\mathcal{H})}^2$ and $\int_{\K\cup\mathcal{H}_\ell}|\varphi_{b}|^{p}\,dx  \leq b^{-1}\int_{\mathcal{H}}|\varphi^1|^{p}\,dx$. Moreover, for large enough $\ell$, one has
\[
    \int_{\K\cup\mathcal{H}_\ell}|\varphi_{b}|^{p}\,dx = b^{-1}\int_{\mathcal{H}}|\varphi^1|^{p}\,dx.
\]
Repeating the arguments of Lemma \ref{lemcin}, we obtain
\begin{equation*}
  \frac{\norm{\varphi_b^+}^2_2}{\norm{\varphi_b}^2_2} \to 1, \quad \frac{\norm{\varphi_b^+}^2}{\norm{\varphi_b}^2} \to 1 \text{ and } b\int_{\G}|\varphi_{b}^+|^{p}\,dx \to b\int_{\G}\abs{\varphi_b}^{p}\,dx = \int_{\mathcal{H}}|\varphi^1|^{p}\,dx\,\, \text{ as }\,\, b \to 0.
\end{equation*}
Repeating the arguments of Lemma \ref{lemcinsup}, we know,
for sufficiently small $b$ , there exists $\ell^*_b$, such that, for all $\ell > \ell^*_b$,
\begin{equation*}
            \begin{aligned}
                J_{s}(\frac{\varphi_b^+}{\norm{\varphi_b^+}_2}) =& \frac{\norm{\varphi_b^+}^2}{2\norm{\varphi_b^+}^2_2}-\frac{1}{2}\left\lVert h_s\left(\frac{\varphi_b^+}{\norm{\varphi_b^+}_2}\right) \right\rVert^2 - \frac{s}{2} - \frac{s}{2}\left\lVert h_s\left(\frac{\varphi_b^+}{\norm{\varphi_b^+}_2}\right) \right\rVert_2^2\\
                &- \Psi_\ell\left(\frac{\varphi_b^+}{\norm{\varphi_b^+}_2}+h_s\left(\frac{\varphi_b^+}{\norm{\varphi_b^+}_2}\right)\right)\\
                \leq& \frac{m-s}{2}+ C_1b^2 - C_2ab^{\frac{p}{2}-1} \\&- \frac{1}{2}\left\lVert h_s\left(\frac{\varphi_b^+}{\norm{\varphi_b^+}_2}\right)\right\rVert^2\left(\min\left\{1, 1 + \frac{\omega_*}{m}\right\} - C_3 \left(\min\left\{1, 1 + \frac{\omega_*}{m}\right\}\right)^\frac{2-p}{2}a^\frac{p}{2}b^\frac{(p-2)^2}{4}\right)
    \end{aligned}
    \end{equation*}
and
$$
\begin{aligned}
g_s'(t)  =&\frac{1}{t}\langle J_s^{\prime}(t\varphi_b^+), t\varphi_b^+\rangle \\
 \geq& t\left(\norm{\varphi_b^+}^2-a\norm{\varphi_b^+}_{2}^{2-p}\int_{\K\cup\mathcal{H}_\ell} \abs{\varphi_b^+}^{p} \,dx\right)\\
 \geq& t(C_1b^{-1} - C_2ab^\frac{p-4}{2}),
\end{aligned}
$$
where $0 \leq t \leq \norm{\varphi_b^+}_{2}^{-1}$. Fix $b$ small enough. Then, since $2< p <6$, there exists $\ell^*$, such that, for all $\ell>\ell^*$, we have $J_s(\frac{\varphi_b^+}{\norm{\varphi_b^+}_2}) < \frac{m-s}{2}$ and $g_s(t)$ is increasing for $0 \leq t \leq\norm{\varphi_b^+}_{2} ^{-1}$. Moreover,
$$
c_{s,\infty} \leq \sup _{0 \leq t \leq\norm{\varphi_b^+}_{2} ^{-1}} J_s(t\varphi_b^+) \leq J_s(\frac{\varphi_b^+}{\norm{\varphi_b^+}_2}) < \frac{m-s}{2}.
$$
This completes the proof of Lemma \ref{lemcin46}.
\end{proof}
Let $4 \leq p <6$. Then, from Theorem \ref{th31}, for $\omega_*=0$ and fixed $a > 0$, we see that, for any $\ell > \ell^*(0,a)$,\\
\begin{enumerate}[label=\rm(\roman*)]
\item either there exist a non-trivial $u_{0} \in \operatorname{dom}(\D)$ and $\omega_{0} \in[0, m)$ such that
$$
\left\{\begin{array}{l}
\D u_{0}-\omega_{0} u_{0}=a\chi_{\ell}\abs{u_{0}}^{p-2} u_{0},\\
\int_{\G}\left|u_{0}\right|^{2} \,dx = 1,
\end{array}\right.
$$
and
$$
0<\frac{1}{2}(\D u_{0}, u_{0})_{2}-\Psi_\ell(u_{0})=c_{\infty}<\frac{m}{2},
$$
\item or there exists a non-trivial $u_{0} \in \operatorname{dom}(\D)$ such that
$$
\left\{\begin{array}{l}
\D u_{0}=a\chi_{\ell}\abs{u_{0}}^{p-2} u_{0}, \\
\int_{\G}\left|u_{0}\right|^{2} \,dx < 1,
\end{array}\right.
$$
and
$$
0<\frac{1}{2}(\D u_{0}, u_{0})_{2}-\Psi_\ell(u_{0})=c_{\infty}<\frac{m}{2}.
$$
\end{enumerate}
 From the following Lemma, we know if $0<a \leq a_{*,0}$ (as defined in \eqref{eqa*0}), then for all $\ell > \ell^*(0,a)$, the second case cannot occur. Specifically, there exist a non-trivial $u_{0} \in \operatorname{dom}(\D)$ and $\omega_{0} \in[0, m)$ satisfying
$$
\left\{\begin{array}{l}
\D u_{0}-\omega_{0} u_{0}=a\chi_{\ell}\abs{u_{0}}^{p-2} u_{0},\\
\int_{\G}\left|u_{0}\right|^{2} \,dx = 1,
\end{array}\right.
$$
which completes the proof of Theorem \ref{th4}.
\begin{lemma}\label{lempgeq4}
    Let $p \geq 4$. If $0 < a \leq a_{*,0}$, then, for all $\ell \geq 0$, there exists no non-trivial $u_{0} \in \operatorname{dom}(\D)$ such that
    $$
\left\{\begin{array}{l}
\D u_{0}=a\chi_{\ell}\abs{u_{0}}^{p-2} u_{0}, \\
\int_{\G}\left|u_{0}\right|^{2} \,dx < 1,
\end{array}\right.
$$
and
$$
0<\frac{1}{2}(\D u_{0}, u_{0})_{2}-\Psi_\ell(u_{0})\leq\frac{m}{2}.
$$
\end{lemma}
\begin{proof}
Assume that there exists a non-trivial $u_{0} \in \operatorname{dom}(\D)$ such that
    $$
\left\{\begin{array}{l}
\D u_{0}=a\chi_{\ell}\abs{u_{0}}^{p-2} u_{0}, \\
\int_{\G}\left|u_{0}\right|^{2} \,dx < 1,
\end{array}\right.
$$
and
$$
0<\frac{1}{2}(\D u_{0}, u_{0})_{2}-\Psi_\ell(u_{0})\leq\frac{m}{2}.
$$
From $\frac{1}{2}(\D u_{0}, u_{0})_{2}-\Psi_\ell(u_{0})\leq\frac{m}{2}$, we know that
\begin{equation}\label{eqi0m2}
        \frac{1}{2}\norm{u^+_{0}}^2 - \frac{1}{2}\norm{u^-_{0}}^2 - \frac{a}{p}\int_{\K\cup\mathcal{H}_\ell}\abs{u_{0}}^p \leq \frac{m}{2}.
\end{equation}
Then, it follows from $\langle I_{{0}}^{\prime}(u_{0}), u_{0}\rangle = 0$ and \eqref{eqi0m2} that
\begin{equation*}
    \frac{a(p-2)}{2p}\int_{\K\cup\mathcal{H}_\ell}\abs{u_{{0}}}^p\, dx   \leq \frac{m}{2},
\end{equation*}
which implies
\begin{equation}\label{eqintp}
\int_{\K\cup\mathcal{H}_\ell}\abs{u_{{0}}}^p\, dx \leq \frac{mp}{a(p-2)}.
\end{equation}
Using the fact that $\langle I_{{0}}^{\prime}(u_{0}), u^+_{0}\rangle = 0$, and applying \eqref{eqintp}, Lemma \ref{lemgnsg} and the H\"older inequality, we obtain
\begin{equation*}
    \begin{aligned}
    \norm{u^+_{0}}^2 &\leq a\int_{\K\cup\mathcal{H}_\ell}\abs{u_{{0}}}^{p-1}\abs{u_{{0}}^+}\, dx \\
    &\leq a(\int_{\K\cup\mathcal{H}_\ell}\abs{u_{{0}}}^p\, dx)^\frac{p-1}{p}(\int_{\K\cup\mathcal{H}_\ell}\abs{u_{{0}}^+}^p\, dx)^\frac{1}{p} \\
    &\leq a\left[\frac{mp }{a(p-2)}\right]^\frac{p-1}{p}\left(\frac{C_{p,\G}}{m}\right)^\frac{1}{p}\norm{u^+_{0}}\\
    &= (C_{p,\G}a)^\frac{1}{p} m^\frac{p-2}{p}\left(\frac{p }{p-2}\right)^\frac{p-1}{p}\norm{u^+_{0}},
    \end{aligned}
\end{equation*}
and thus, by $\frac{1}{2}\norm{u^+_0}^2 - \frac{1}{2}\norm{u^-_0}^2 \geq I_0(u_0) > 0$, we know
\begin{equation}\label{eqnormm2}
\norm{u_0} < \sqrt{2}\norm{u^+_0} \leq \sqrt{2}(C_{p,\G}a)^\frac{1}{p} m^\frac{p-2}{p}\left(\frac{p}{2p-4}\right)^\frac{p-1}{p}.
\end{equation}
Therefore, similar to the arguments of \eqref{eqi+-}, by Lemma \ref{lemgnsg}, we obtain
\begin{equation}
    \begin{aligned}\label{eqnormu2a}
\norm{u_{{0}}}^2 & \leq a (\int_{\K\cup\mathcal{H}_\ell}\abs{u_{{0}}}^p\, dx)^\frac{p-1}{p}\left[(\int_{\K\cup\mathcal{H}_\ell}\abs{u_{{0}}^+}^p\, dx)^\frac{1}{p} + (\int_{\K\cup\mathcal{H}_\ell}\abs{u_{{0}}^-}^p\, dx)^\frac{1}{p}\right] \\
& \leq C_{p,\G}a(\norm{u_{0}}^{p-2}\norm{u_{0}}_2^2)^\frac{p-1}{p}\left[(\norm{u_{0}^+}^{p-2}\norm{u_{0}^+}_2^2)^\frac{1}{p} + (\norm{u_{0}^-}^{p-2}\norm{u_{0}^-}_2^2)^\frac{1}{p}\right]\\
& \leq 2C_{p,\G}a\norm{u_{0}}^{p-2}\norm{u_{0}}_2^2.
\end{aligned}
\end{equation}
Thus, by \eqref{eqnormm2} and \eqref{eqnormu2a}, we conclude
\[
\begin{aligned}
        \norm{u_{0}}_2^2 &\geq (2C_{p,\G}a\norm{u_{0}}^{p-4})^{-1}\\
        & > 2^{1-\frac{1}{2}p}(C_{p,\G}a)^{-1}(C_{p,\G}a)^\frac{4-p}{p} m^\frac{-p^2+6p-8}{p}\left(\frac{p}{p-2}\right)^\frac{-p^2+5p-4}{p}\\
        & \geq 2^{1-\frac{1}{2}p}a^\frac{4-2p}{p}C_{p,\G}^\frac{4-2p}{p}m^\frac{-p^2+6p-8}{p}\left(\frac{p}{p-2}\right)^\frac{-p^2+5p-4}{p},
\end{aligned}
\]
which leads to a contradiction when $a \leq 2^{-\frac{1}{4}p}C_{p,\G}^{-1}m^\frac{4-p}{2}\left(\frac{p}{p-2}\right)^\frac{-p^2+5p-4}{2p-4} = a_{*,0}$.
\end{proof}
Now we turn to equation \eqref{eq1} on $\G_\ell$ and provide the proof of Corollary \ref{co1}.
\begin{proof}[Proof of Corollary \ref{co1}]
It suffices to show that if $u$ is a solution of the equation
\[
\D_e u_e - \omega u_e= a\chi_\ell\abs{u_e}^{p-2}u_e \quad \forall e\in \mathrm{E},
\]
 then there exists a solution $u_\ell$ of the equation
\[
\D_{\ell,e} u_{\ell,e} - \omega u_{\ell,e}= a\chi_{\K_\ell}\abs{u_{\ell,e}}^{p-2}u_{\ell,e} \quad \forall e\in \mathrm{E}_\ell,
\]
with $\int_{\G_\ell}\abs{u_\ell}^2\, dx = \int_{\G}\abs{u}^2\, dx$.

Assume that $u$ is a solution of equation
\[
\D_e u_e - \omega u_e= a\chi_\ell\abs{u_e}^{p-2}u_e\quad \forall e\in \mathrm{E}.
\]

Let $\mathcal{H}_\ell$ denote the half-line incident to $v_\ell$ in $G_\ell$, and identity the unbounded edges $\mathcal{H}$ and $\mathcal{H}_\ell$ with $[0,\infty)$, while identifying the bounded edge $v_\mathcal{H}v_\ell$ of $\G_\ell$ with the compact intervals $[0,\ell]$, where $x_{v_\mathcal{H}v_\ell}=0$ at $v_\mathcal{H}$. Define $u_\ell: \G_\ell \to \mathbb{C}^2$ by
$$
    u_\ell(x)= \begin{cases}u_{\mathcal{H}}(x) & \text { for }  x\in [0,\ell] = v_\mathcal{H}v_\ell, \\u_{\mathcal{H}}(x+\ell) & \text { for } x \in [0,+\infty)=\mathcal{H}_\ell,\\u(x)& \text { for } x \in \G_\ell\backslash(v_\mathcal{H}v_\ell \cup \mathcal{H}_\ell), \end{cases}
$$
where $u_{\mathcal{H}}$ is the restriction of $u$ to the edge ${\mathcal{H}}$. Since $u \in \operatorname{dom}(\D)$, it follows that $u_\ell \in \operatorname{dom}(\D_\ell)$ and satisfies
$$\int_{\G_\ell}\abs{u_\ell}^2\, dx = \int_{\G}\abs{u}^2\, dx$$
 and that $u_\ell$ solves the equation
\[
\D_{\ell,e} u_{\ell,e} - \omega u_{\ell,e}= a\chi_{\K_\ell}\abs{u_{\ell,e}}^{p-2}u_{\ell,e} \quad \forall e\in \mathrm{E}_\ell
\]
in the weak sense.
\end{proof}

\section{The negative case}\label{sectnegative}
In this section, we consider the following (NLDE) on $\G$
\begin{equation}
\label{eq-a}
    \D u + \omega u= -a\chi_\K\abs{u}^{p-2}u,
\end{equation}
where $p > 2$, $a > 0$, $\omega\in (-mc^2,mc^2)$ and $\D$ is defined in Definition \ref{defD}.

It is clear that if  $u = (u^1, u^2)^T$ is a solution of equation \eqref{eq-a}, then $(u^2,-u^1)^T$ is a solution of equation
\begin{equation*}
    -\imath c \sigma_1 u_e^{\prime}+m c^2 \sigma_3 u_e - \omega u_e=a\chi_\K\abs{u_e}^{p-2}u_e \quad \forall e \in \mathrm{E}.
\end{equation*}
We note that $(u^2,-u^1)^T$ may not belong to $\operatorname{dom}(\D)$ due to the Kirchhoff-type vertex conditions \eqref{eqdefdom1} and \eqref{eqdefdom2}.
Similar to Theorem \ref{th1}-\ref{th3}, Theorem \ref{th4} and Corollary \ref{co1}, we have the following two results.
\begin{theorem}\label{th-a1}
      Let $\G$ be any noncompact metric graph with a non-empty compact core $\K$. Assume that $\G$ has two different vertices incident with half-lines. Then Theorem \ref{th1}, \ref{th2} and \ref{th3} hold for equation \eqref{eq-a} with $2<p<4$, and Theorem \ref{thalla} hold for equation \eqref{eq-a} with $p > 2$.
\end{theorem}
We define the energy functional $\tilde{I_{\omega}}: Y \to \mathbb{R}$ associated with equation \eqref{eq-a} by
$$
\widetilde{I_{\omega}}(u) := \frac{1}{2}\norm{u^{-}}^{2}-\frac{1}{2}\norm{u^{+}}^{2}-\frac{\omega}{2} \int_{\G}\abs{u}^{2} \,dx-\Psi(u).
$$
Except for Lemma \ref{lemcin}, the proof of Theorem \ref{th-a1} follows exactly the same line as that of Theorem \ref{th1}-\ref{th3} after we exchange the position of $u^+$ with $u^-$.
 Let the two different vertices of $\G$ incident with half-lines be denoted by $v$ and $w$. Since $\G$ is connected, we can denote the simple path connecting $v$ and $w$ by $\{v_k\}_{k=0}^n$, where $v_0=v$ and $v_n=w$. Without loss of generality, we can assume that the terminal vertex of $v_iv_{i+1}$ is the origin vertex of $v_{i+1}v_{i+2}$ for $i=0,1,...,n-2$. Denote the half-line incident at $v$ by $\mathcal{H}_v$ and the half-line incident at $w$ by $\mathcal{H}_w$. Then we can define $\varphi_{b}^2: \G \to \mathbb{C}$ as
$$
    \varphi_{b}^2(x)= \begin{cases}1 & \text { for } x \in v_iv_{i+1},\quad \quad i = 1,2,...,n-1,\\ -\max\{0,1-bx\} & \text { for } x \in \h_v,\\ \max\{0,1-bx\} & \text { for } x \in \h_w,\\0 &\text { else},  \end{cases}
$$
and define the test function $\varphi_{b}: \G \to \mathbb{C}^2$ by $\varphi_{b} = (0,\varphi_{b}^2)^T$.
Then Lemma \ref{lemcin} remains valid after making obvious changes (in particular, $\D-mI$ and $\left(0, -\imath \left(\varphi_{b}^1\right)^\prime\right)^T$ should be replaced by $\D + mI$ and $\left( -\imath \left(\varphi_{b}^2\right)^\prime,0\right)^T$, respectively).
\begin{theorem}\label{th-a2}
      Let $\G$ be any noncompact metric graph with a non-empty compact core $\K$. Then Theorem \ref{th4} and Corollary \ref{co1} hold for equation \eqref{eq-a} with $4\leq p<6$. Moreover, by replacing $a_{*,0}$ with $
\frac{m^\frac{4-p}{2}c^{4-p}}{2C_{p,\G}}
$, Theorem \ref{th4} and Corollary \ref{co1} hold for equation \eqref{eq-a} with $2< p<4$.
\end{theorem}
The proof of the case $4\leq p<6$ in Theorem \ref{th-a2} follows exactly the same line as that of Theorem \ref{th4} and \ref{co1} after we exchange the position of $u^+$ with $u^-$ and replace $\varphi_{b} = (\varphi_{b}^1,0)^T$ by $\varphi_{b} = (0,\varphi_{b}^2)^T$, where $\varphi_{b}^2:=\varphi_{b}^1$. For the case $2< p<4$, we only need to establish a non-existence result, similar to the proof of Theorema \ref{th2}, with $C_{p,\K}$ and $\int_{\K}|u|^{p}\,dx$ replaced by $C_{p,\G}$ and $\int_{\K\cup\mathcal{H}_\ell}|u|^{p}\,dx$, respectively.

Now, we state an assumption and present the main result in this section.\\
{\rm (A)} $\lambda = -mc^2$ is an eigenvalue of the operator $\D$ (as defined in Definition \ref{defD}).

Set $$
    \widetilde{a_0}= \begin{cases}a_0 & \text { for } p\in(2,4),\\ a_{*,0}C_{p,\G}C_{p,\K}^{-1} & \text { for } p \in [4,+\infty),  \end{cases}
$$
where $a_0$ is defined by \eqref{eqa0} and $a_{*,0}$ is defined by \eqref{eqa*0}.
\begin{theorem}\label{th-a3}
     Let $\G$ be any noncompact metric graph with a non-empty compact core $\K$. Assume that {\rm(A)} holds, and let $p>2$. Then, one of the following two cases holds\\
    \begin{enumerate}[label=\rm(\roman*)]
        \item either there exists a non-trivial $u \in \operatorname{dom}(\D)$ (as defined in Definition \ref{defD}) and $\omega \in (-mc^2, mc^2)$ satisfy
    \begin{equation*}
    \left\{\begin{aligned}
        &\D_e u_e + \omega u_e=-a\chi_\K\abs{u_e}^{p-2}u_e \quad \forall e\in \mathrm{E},\\
        &\int_\G\abs{u}^2\,dx = 1.
    \end{aligned}
    \right.
\end{equation*}
\item  or for all $\omega \in (-mc^2, mc^2)$, there exists a non-trivial $u_\omega \in \operatorname{dom}(\D)$ such that
    \begin{equation*}
    \left\{\begin{aligned}
        &\D_e u_{\omega,e} + \omega u_{\omega,e} = -a\chi_\K\abs{u_{\omega,e}}^{p-2}u_{\omega,e} \quad\forall e\in \mathrm{E},\\
        &\int_\G\abs{u_\omega}^2\,dx < 1.
    \end{aligned}
    \right.
    \end{equation*}
\end{enumerate}
Moreover, if $0<a<\widetilde{a_0}$, then the first case above occurs with a  non-trivial $u \in \operatorname{dom}(\D)$ (as defined in Definition \ref{defD}) and $\omega \in [0, mc^2)$.
That is, $u$ and $\omega$ satisfy
    \begin{equation*}
    \left\{\begin{aligned}
        &\D_e u_e + \omega u_e= -a\chi_\K\abs{u_e}^{p-2}u_e \quad \forall e\in \mathrm{E},\\
        &\int_\G\abs{u}^2\,dx = 1.
    \end{aligned}
    \right.
\end{equation*}
\end{theorem}
\begin{remark}\label{remarknegative}
    If $\K$ has a simple cycle (as defined in \ref{defcomg}), then $\lambda = -mc^2$ is an eigenvalue of the operator $\D$ (as defined in Definition \ref{defD}). Indeed, let the simple cycle denoted by $\{v_k\}_{k=0}^n$, where $v_0=v_n$. Without loss of generality, we can assume that the terminal vertex of $v_iv_{i+1}$ is the origin vertex of $v_{i+1}v_{i+2}$ for $i=0,1,...,n-2$ and the terminal vertex of $v_{n-1}v_{n}$ is $v_n$. Define $\varphi^2: \G \to \mathbb{C}$ as
    $$
    \varphi^2(x)= \begin{cases}1 & \text { for } x \in v_iv_{i+1},\quad \quad i = 1,2,...,n-1,\\0 &\text { else},  \end{cases}
$$
and define $\varphi: \G \to \mathbb{C}^2$ by $\varphi = (0,\varphi^2)^T$. Then $\varphi$ is the eigenfunction corresponding to the eigenvalue $\lambda = -mc^2$.

If $\lambda = mc^2$ is an eigenvalue of the operator $\D$ (as defined in Definition \ref{defD}), then Theorem \ref{th-a3} remains valid for equation \eqref{eq1}. However, the specific metric graph for which $\lambda = mc^2$ is an eigenvalue of the operator $\D$ (as defined in Definition \ref{defD}) is unknown.

Moreover, if {\rm(A)} is satisfied and $\K$ is a tree with at most one leaf incident with no half-line in $\G$, then Theorem \ref{thalla} remains valid for equation \eqref{eq-a} with $p>2$. However, the metric graph for which both assumptions are  satisfied remains unknow, as the assumption that $\K$ is a tree contradicts the presence of simple cycle.
\end{remark}
As in Section \ref{subsecper}, by exchanging the position of $u^+$ with $u^-$, we define $\widetilde{J}: Y^{-} \rightarrow \mathbb{R}$ by
$$
\widetilde{J}(v)=\widetilde{I_0}(v+\widetilde{h}(v))=\max \{\widetilde{I_0}(v+w): w \in Y^{+}\},
$$
where $\widetilde{h} \in C^1\left(Y^{-}, Y^{+}\right)$ and the critical points of $\widetilde{J}$ and $\widetilde{I_0}$ are in one-to-one correspondence via the injective map $u \rightarrow u+\widetilde{h}(u)$ from $Y^{-}$into $Y$. Define  $\widetilde{I_{r, \mu}}(u):U_\mu \to \mathbb{R}$ by
$$
\widetilde{I_{r, \mu}}(u) := \frac{1}{2}\norm{u^{-}}^{2}-\frac{1}{2}\norm{u^{+}}^{2}-\Psi(u)-H_{r, \mu}(u),
$$ and
$\widetilde{J_{r, \mu}}: Y^- \cap U_\mu \to \mathbb{R}$ by
$$
\widetilde{J_{r, \mu}}(v):=\max \left\{\widetilde{I_{r, \mu}}(v+w): w \in Y^{+},(T_{\mu}(v+w), v+w)_{2}<1\right\}.
$$
Define
$$
\widetilde{c_{\infty}}:=\sup _{ r>1, \mu>0} \widetilde{c_{r, \mu}},
$$
where
$$\widetilde{c_{r, \mu}} := \inf _{\gamma \in \Gamma_{r, \mu}} \max _{t \in[0,1]} \widetilde{J_{r, \mu}}(\gamma(t))>0,$$
$$
\widetilde{\Gamma_{r, \mu}} :=\{\gamma \in C([0,1], Y^{-} \cap U_{\mu}): \gamma(0)=0, \widetilde{J_{ r, \mu}}(\gamma(1))<0\}.
$$
Except for the upper bound of $\widetilde{c_\infty}$, the proof of the first part of Theorem \ref{th-a3} follows exactly the same line as that of Theorem \ref{th1} after we exchange the position of $u^+$ with $u^-$. For the second part of Theorem \ref{th-a3}, we only need to establish a non-existence result as in the proof of Theorem \ref{th2} and Lemma \ref{lempgeq4} ($C_{p,\G}$ and $\int_{\K\cup\mathcal{H}_\ell}|u|^{p}\,dx$ should be replaced by $C_{p,\K}$ and $\int_{\K}|u|^{p}\,dx$, respectively). Thus, to prove Theorem \ref{th-a3}, it suffices to provide the upper bound of $\widetilde{c_\infty}$. In the following two lemmas, we still set $c=1$ for simplicity.
\begin{lemma}\label{lemeg}
    Assume that {\rm(A)} holds. Then there exists $\varphi \in Y^-\backslash\{0\}$ such that
    \begin{equation*}
        \norm{\varphi}^2=m\norm{\varphi}_2^2.
    \end{equation*}
Moreover, $\varphi$ does not vanish identically on $\K$.
\end{lemma}
\begin{proof}
      Denote the set of projections of $L^2(\G, \mathbb{C}^2)$ associated with the spectral measure $\mu_u^{\mathcal{D}}$ by $$\{\mu^{\mathcal{D}}(\triangle) : \triangle\subset \mathbb{R} \text{ is a Borel set}\}.$$
      It is well known that $P^+=\mu^{\mathcal{D}}((0, +\infty))$ and $P^-=\mu^{\mathcal{D}}((-\infty , 0))$.

      Since {\rm(A)} holds, we know that there exists $\varphi \in Y\backslash\{0\}$ such that
    \begin{equation*}
        \D\varphi = -m\varphi,
    \end{equation*}
    hence
    \begin{equation*}
            \norm{\varphi^+}^2 - \norm{\varphi^-}^2 =-m\norm{\varphi}_2^2.
    \end{equation*}
We need to show that $\varphi \in Y^-$. Let
    \[
  g_n(x):=
  \begin{cases}
    \frac{1}{m+x}, \quad  & x \in \sigma(\D)\backslash(-m-\frac{1}{n},-m+\frac{1}{n}),\\
    0,     & x \in (-m-\frac{1}{n},-m+\frac{1}{n}).
\end{cases}
\]
Then, for each $n \in\mathbb{N}$, $\{g_n\}$ is a bounded and Borel measurable function on $\sigma(\D)$\blue{,} and
\[
    g_n(\D)(mI+\D) = \mu^{\mathcal{D}}\left(\mathbb{R}\backslash(-m-\frac{1}{n},-m+\frac{1}{n})\right).
\]
Therefore,
\[
    \mu^{\mathcal{D}}\left(\mathbb{R}\backslash(-m-\frac{1}{n},-m+\frac{1}{n})\right)\varphi = g_n(\D)(mI+\D)\varphi =0,
\]
and thus, for $n$ sufficiently large, we have
\[
    P^+\varphi = \mu^{\mathcal{D}}((0,+\infty))\varphi =  \mu^{\mathcal{D}}((0,+\infty))\mu^{\mathcal{D}}\left(\mathbb{R}\backslash(-m-\frac{1}{n},-m+\frac{1}{n})\right)\varphi = 0.
\]

 Now, we show that $\varphi$ does not vanish identically on $\K$. By \cite[Remark 6 in Chapter 8]{Ha}, we get $u_e\in C^1(I_e,\mathbb{C}^2)$ for all $e \in \mathrm{E}$. Using \eqref{eqdefde}, we see $u_{\mathcal{H}} \equiv 0$ on every half-line $\mathcal{H}\in\mathrm{E}$. Then $\varphi$ does not vanish identically on $\K$, otherwise $\varphi \equiv 0$.
\end{proof}
\begin{lemma}\label{lemcin-a}
    Assume that {\rm(A)} holds. Then $\widetilde{c_\infty} < \frac{m}{2}$.
\end{lemma}
\begin{proof}
    Let $\varphi$ defined as in Lemma \ref{lemeg}. Since $\widetilde{h} \in C^1(Y^-,Y^+)$ and $\widetilde{h}(0) = 0$, we have $\widetilde{h}(t\varphi) \to 0$ as $t \to 0$. Thus, for $t$ sufficiently small, the Sobolev inequality and convexity of $\Psi$ implies
    \begin{equation}\label{eqjsmallm-a}
            \begin{aligned}
                \widetilde{J}(t\varphi) &= \frac{1}{2}\norm{t\varphi}^2-\frac{1}{2}\norm{\widetilde{h}(t\varphi)}^2 - \Psi(t\varphi+\widetilde{h}(t\varphi))\\
                &\leq \frac{1}{2}\norm{t\varphi}^2-\frac{1}{2}\norm{\widetilde{h}(t\varphi)}^2 - 2^{1-p}\Psi(t\varphi) + \Psi(-\widetilde{h}(t\varphi))\\
                &\leq \frac{m}{2}\norm{t\varphi}_2^2 - 2^{1-p}\Psi(t\varphi) -\frac{1}{2} \norm{\widetilde{h}(t\varphi)}^2 + C\norm{\widetilde{h}(t\varphi)}^p\\
                &< \frac{m}{2}\norm{t\varphi}_2^2.
    \end{aligned}
    \end{equation}
For $s>0$, define $\widetilde{g}(s):=\widetilde{J}(st\varphi)$, then, from the definition of $\widetilde{J}$, we know
$$
\begin{aligned}
\widetilde{g}'(s) & =\frac{1}{s}\langle \widetilde{J}^{\prime}(st\varphi), st\varphi\rangle=\frac{1}{s}\langle \widetilde{I_0}'(st\varphi+\widetilde{h}(st\varphi)), s \varphi+\widetilde{h}'(st\varphi) st\varphi\rangle \\
& =\frac{1}{s}\langle \widetilde{I_0}^{\prime}(st\varphi+\widetilde{h}(st\varphi)), st\varphi + \widetilde{h}(st\varphi)\rangle \\
& =\frac{1}{s}\left[\norm{st\varphi}^2-\norm{\widetilde{h}(st\varphi)}^2-a\int_{\K} \abs{st\varphi + \widetilde{h}(st\varphi)}^p \,dx\right].
\end{aligned}
$$
Remark that, from the definition of $\widetilde{h}$, we have
$$
\widetilde{I_0}(st\varphi+\widetilde{h}(st\varphi)) \geq \widetilde{I_0}(st\varphi),
$$
which implies that
$$
a\int_{\K}\abs{st\varphi+\widetilde{h}(st\varphi)}^{p}\,dx \leq a\int_{\K}\abs{st\varphi}^{p} \,dx-\frac{p}{2}\norm{\widetilde{h}(st\varphi)}^{2}.
$$
Thus,
$$
\begin{aligned}
\widetilde{g}'(s) & \geq \frac{1}{s}\left[\norm{st\varphi}^{2}-a\int_{\K}\abs{st\varphi}^{p}\,dx+\frac{p-2}{2}\norm{\widetilde{h}(st\varphi)}^{2}\right] \\
& \geq s\norm{t\varphi}^2-as^{p-1}\int_{\K} \abs{t\varphi}^{p} \,dx \\
& \geq s\norm{t\varphi}^2-Cas^{p-1}\norm{t\varphi}^{p}
\end{aligned}
$$
which implies that $\widetilde{g}$ is increasing for $0 \leq s \leq 1$ and $t$ small enough. Fix $t$ sufficiently small. Then, for $1 \leq s \leq \norm{t\varphi}_{2}^{-1}$, we obtain that $\Psi(st\varphi+\widetilde{h}(st\varphi)) \geq s^{2} \Psi(t\varphi+s^{-1} \widetilde{h}(st\varphi))$ and then by \eqref{eqjsmallm-a}
$$
\widetilde{J}(st\varphi)=\widetilde{I_{0}}(st\varphi+\widetilde{h}(st\varphi)) \leq s^{2} \widetilde{I_{0}}(t\varphi+s^{-1} \widetilde{h}(st\varphi)) \leq s^{2} \widetilde{J}(t\varphi)<\frac{m s^{2}}{2}\norm{st\varphi}_{2}^{2} \leq \frac{m}{2}.
$$
Hence, by the above estimates, as in \eqref{eqcssup},
$$
\widetilde{c_{\infty}} \leq \sup _{0 \leq s \leq\norm{t\varphi}_{2} ^{-1}} \widetilde{J}(st\varphi)<\frac{m}{2}.
$$ This completes the proof of Lemma \ref{lemcin-a}.
\end{proof}
\appendix \section{}\label{appa}
In the previous discussion, we studied the influence of the parameter $a>0$ on the existence of solutions to the equations \eqref{eq1} and \eqref{eq-a}. However, since $C_{p,\G}$ and $C_{p,\K}$ in Lemma \ref{lemgnsg} depend on $m,c >0$, the influence of $m,c >0$ on the existence of solutions to the equation \eqref{eq1} and \eqref{eq-a} remains unclear. In this appendix, we establish a non-existence result for $2<p<6$ and  investigate the influence of the parameters $m, c>0$ on the existence of normalized solutions to equations \eqref{eq1} and \eqref{eq-a}.

In the following, we abbreviate $\norm{u}_{H^1(\G, \mathbb{C}^2)}$ as $\norm{u}_{H^1}$. The following lemma is an immediate corollary of the Gagliardo-Nirenberg-Sobolev inequality \cite[Theorem 1.1]{Ha2}.
\begin{lemma}\label{lemgnsgh1}
Let $p \geq 2$. Then there exist constants $S_{p,\K}$ and $S_{p,\G}$ that depend only on $p$ and $\G$ such that
\[
    \int_{\K} \abs{u}^p\,dx \leq S_{p,\K} \norm{u}_{H^1}^{\frac{p}{2}-1}\norm{u}_2^{\frac{p}{2}+1}, \quad \forall u \in H^1(\G, \mathbb{C}^2),
\]
\[
    \int_{\G} \abs{u}^p\,dx \leq S_{p,\G} \norm{u}_{H^1}^{\frac{p}{2}-1}\norm{u}_2^{\frac{p}{2}+1}, \quad \forall u \in H^1(\G, \mathbb{C}^2),
\]
and there exist constants $S_{\infty,\K}$ and $S_{\infty,\G}$ that depend only on $\G$ such that
\[
 \norm{u}_{L^\infty(\K, \mathbb{C}^2)} \leq S_{\infty,\K} \norm{u}_{H^1}^\frac{1}{2}\norm{u}_2^\frac{1}{2}, \quad \forall u \in H^1(\G, \mathbb{C}^2),
\]
\[
 \norm{u}_{L^\infty(\G, \mathbb{C}^2)} \leq S_{\infty,\G} \norm{u}_{H^1}^\frac{1}{2}\norm{u}_2^\frac{1}{2}, \quad \forall u \in H^1(\G, \mathbb{C}^2).
\]
\end{lemma}
\begin{lemma}\label{lempin26}
    For $2<p < 4$, if
    \begin{equation}\label{eqapppin24}
         0<a\sqrt2S_{2p-2,\K}\max\{c^{-1},(mc^2)^{-1}\}\leq 1,
    \end{equation}
then there exists no non-trivial $u_{0} \in \operatorname{dom}(\D)$ such that
    $$
\left\{\begin{array}{l}
\D u_{0}=a\chi_{\ell}\abs{u_{0}}^{p-2} u_{0}, \\
\int_{\G}\left|u_{0}\right|^{2} \,dx < 1.
\end{array}\right.
$$
Moreover, for $2< p < 6$, if
    \begin{equation}\label{eqapppin46}
         0<a S_{\infty,\K}^\frac{(p-2)^2}{4}S_{2p-2,\K}^\frac{6-p}{8}\left(\frac{2p}{p-2}\right)^\frac{p-2}{4} \max\{m^\frac{p-4}{2}c^\frac{p-6}{2},(mc^2)^{-1}\}\leq 1,
    \end{equation}
then there exists no non-trivial $u_{0} \in \operatorname{dom}(\D)$ such that
    $$
\left\{\begin{array}{l}
\D u_{0}=a\chi_{\ell}\abs{u_{0}}^{p-2} u_{0}, \\
\int_{\G}\left|u_{0}\right|^{2} \,dx < 1,
\end{array}\right.
$$
and
$$
0<\frac{1}{2}(\D u_{0}, u_{0})_{2}-\Psi_\ell(u_{0})\leq \frac{mc^2}{2}.
$$
\end{lemma}
\begin{proof}
Assume that there exists a non-trivial $u_{0} \in \operatorname{dom}(\D)$ such that
    $$
\left\{\begin{array}{l}
\D u_{0}=a\chi_{\ell}\abs{u_{0}}^{p-2} u_{0}, \\
\int_{\G}\left|u_{0}\right|^{2} \,dx < 1.
\end{array}\right.
$$
By \cite[Remark 6 in Chapter 8]{Ha}, we obtain $u_{0,e}\in C^1(I_e,\mathbb{C}^2)$ for all $e \in \mathrm{E}$. Then, a direct computation yields that
\begin{equation}\label{eqappdu2}
    a^2\int_{\K}\abs{u_{{0}}}^{2p-2}\, dx= \norm{\D u}_2^2=c^2\norm{u'}_2^2 + m^2c^4\norm{u}_2^2\geq \frac{\min\{c^2,m^2c^4\}}{2}\norm{u}_{H^1}^2,
\end{equation}
and thus, by Lemma \eqref{lemgnsgh1}, we obtain
\[\begin{aligned}
    \frac{\min\{c^2,m^2c^4\}}{2}\norm{u}_{H^1}^2&\leq a^2 \int_{\K}\abs{u_{{0}}}^{2p-2}\, dx \\
    &\leq a^2S_{2p-2,\K}\norm{u}_{H^1}^{p-2}\norm{u}_2^p.
\end{aligned}
\]
Hence, for $2<p<4$,
\[
    \frac{\min\{c^2,m^2c^4\}}{2}\norm{u}_2^{4-p}\leq \frac{\min\{c^2,m^2c^4\}}{2}\norm{u}_{H^1}^{4-p} \leq a^2S_{2p-2,\K}\norm{u}_2^p,
\]
that is $a\sqrt2S_{2p-2,\K}\max\{c^{-1},(mc^2)^{-1}\}>1$,  a contradiction to \eqref{eqapppin24}.

Now, suppose  $u_0$ satisfies
$$
0<\frac{1}{2}(\D u_{0}, u_{0})_{2}-\Psi_\ell(u_{0})\leq \frac{mc^2}{2},
$$
that is
\begin{equation}\label{eqappi0m2}
        \frac{1}{2}\norm{u^+_{0}}^2 - \frac{1}{2}\norm{u^-_{0}}^2 - \frac{a}{p}\int_{\K}\abs{u_{0}}^p \leq \frac{mc^2}{2}.
\end{equation}
Then, it follows from $\langle I_{{0}}^{\prime}(u_{0}), u_{0}\rangle = 0$ and \eqref{eqappi0m2} that
\begin{equation*}
    \int_{\K}\abs{u_{{0}}}^p\, dx   \leq \frac{pmc^2}{a(p-2)},
\end{equation*}
and thus, by Lemma \eqref{lemgnsgh1}, we obtain
\[\begin{aligned}
    \frac{\min\{c^2,m^2c^4\}}{2}\norm{u}_{H^1}^2&\leq a^2 \int_{\K}\abs{u_{{0}}}^{2p-2}\, dx \\
    &\leq a^2\norm{u_0}_{L^\infty(\K, \mathbb{C}^2)}^{p-2}\int_{\K}\abs{u_{{0}}}^{p}\, dx\\
    &\leq\frac{apmc^2}{(p-2)}S_{\infty,\K}^{p-2} \norm{u}_{H^1}^\frac{p-2}{2}\norm{u}_2^\frac{p-2}{2}.\\
    &\leq\frac{apmc^2}{(p-2)}S_{\infty,\K}^{p-2} \norm{u}_{H^1}^\frac{p-2}{2},\\
\end{aligned}
\]
which implies
\[
\norm{u}_{H^1}^{3-\frac{p}{2}}\leq \max\{ m,\frac{1}{mc^2}\}\frac{2ap}{p-2}S_{\infty,\K}^{p-2}.
\]
It follows from \eqref{eqappdu2} and Lemma \eqref{lemgnsgh1} that
\[\begin{aligned}
m^2c^4\norm{u}_2^2&\leq a^2 \int_{\K}\abs{u_{{0}}}^{2p-2}\, dx\\
&\leq a^2S_{2p-2,\K}\norm{u}_{H^1}^{p-2}\norm{u}_2^p\\
&\leq a^\frac{8}{6-p}\left(\frac{2p}{p-2}\max\{ m,\frac{1}{mc^2}\}\right)^\frac{2p-4}{6-p}S_{\infty,\K}^\frac{2(p-2)^2}{6-p}S_{2p-2,\K}\norm{u}_2^p,
\end{aligned}
\]
hence
\[
 a S_{\infty,\K}^\frac{(p-2)^2}{4}S_{2p-2,\K}^\frac{6-p}{8}\left(\frac{2p}{p-2}\right)^\frac{p-2}{4} \max\{m^\frac{p-4}{2}c^{\frac{p-6}{2}},(mc^2)^{-1}\}> 1,
\]
 which is a contradiction to \eqref{eqapppin46} and this completes the proof.
\end{proof}
    We now present some conclusions derived from Lemma \ref{lempin26}. For $2<p<4$,  the condition that $0<a<a_{0}$ in Theorem \ref{th2}, \ref{th-a1} and \ref{th-a3} can be replaced by \eqref{eqapppin24}. For $4\leq p<6$, the condition that $0<a<a_{*,0}C_{p,\G}C_{p,\K}^{-1}$ in Remark \ref{remarkp>4} and Theorem \ref{th-a3} can be replaced by \eqref{eqapppin46}. Moreover, for $2<p<6$, the condition $0<a<a_{*,0}$ in Theorem \ref{th4}, Corollary \ref{co1} and Theorem \ref{th-a2} can be replaced by the following inequality
    \begin{equation*}\label{eqapppin26g}
         0<a S_{\infty,\G}^\frac{(p-2)^2}{4}S_{2p-2,\G}^\frac{6-p}{8}\left(\frac{2p}{p-2}\right)^\frac{p-2}{4} \max\{m^\frac{p-4}{2}c^\frac{p-6}{2},(mc^2)^{-1}\}\leq 1.
    \end{equation*}
    Roughly speaking, for $2<p<4$, there exists a constant $c_0>0$ that depends only on $a$, $m$ and $\G$ such that if $c>c_0$, then equation \eqref{eq1} (or \eqref{eq-a} with $\K$ having a simple cycle, see Definition \ref{defcomg}) has a normalized solution $u_c$ with $\omega_c \in (-mc^2,mc^2)$ and $\int_{\G}\abs{u_c}^{2}\, dx=1$. However, this is insufficient for studying the nonrelativistic limit of $u_c$, as we need a uniform upper bound of $C_{p,\K}$ with respect to $c>0$ to  demonstrate that $mc^2-\omega_c\not \to0$ as $c \to +\infty$, see \cite[Theorem 2.13]{Bo} for more details.

\subsection*{Conflict of interest}

On behalf of all authors, the corresponding author states that there is no conflict of interest.

\subsection*{Ethics approval}
 Not applicable.

\subsection*{Data Availability Statements}
Data sharing not applicable to this article as no datasets were generated or analysed during the current study.

\subsection*{Acknowledgements}
 C. Ji was partially supported by National Natural Science Foundation of China (No. 12171152).

  \end{document}